\documentclass{scrartcl}
\usepackage[utf8]{inputenc}
\usepackage[T1]{fontenc}

\usepackage{amsthm}
\usepackage{amsfonts}
\usepackage{stmaryrd}
\usepackage{MnSymbol}
\usepackage{xspace}
\usepackage{tikz-cd}
\usepackage[french=guillemets]{csquotes}
\MakeOuterQuote{"}
\usepackage[shortlabels]{enumitem}
\setlist{nosep}
\usepackage[style=alphabetic,backend=biber,sorting=nyt,giveninits=true]{biblatex}
\usepackage[final]{microtype}

\usepackage[colorlinks=true, breaklinks=true, urlcolor= black, linkcolor= black, citecolor= black, 
bookmarksopen=true,linktocpage=true,plainpages=false,pdfpagelabels]{hyperref}
\usepackage[capitalise]{cleveref}

\newtheorem{theorem}{Theorem}[section]
\newtheorem*{theorem*}{Theorem}
\newtheorem{claim}{Claim}[theorem]
\newtheorem{lemma}[theorem]{Lemma}
\newtheorem{corollary}[theorem]{Corollary}
\newtheorem{proposition}[theorem]{Proposition}

\theoremstyle{definition}
\newtheorem{remark}[theorem]{Remark}
\newtheorem{definition}[theorem]{Definition}
\newtheorem{notation}[theorem]{Notation}

\newcommand\Hen[1][]{\mathrm{Hen}_{#1}}

\newcommand\ACVF{\ensuremath{\mathrm{ACVF}}}

\newcommand\K{\mathrm{K}}
\renewcommand\L{\mathfrak{L}}
\newcommand{\eq}[1]{#1^{\mathrm{eq}}}
\newcommand{\acl}[1][]{\mathrm{acl}\ifstrempty{#1}{}{_{#1}}}
\newcommand{\dcl}[1][]{\mathrm{dcl}\ifstrempty{#1}{}{_{#1}}}
\newcommand{\acleq}[1][]{\acl[#1]}
\newcommand{\dcleq}[1][]{\dcl[#1]}
\newcommand\code[1]{\ulcorner#1\urcorner}
\newcommand\G{\mathcal{G}}
\newcommand\aut{\mathrm{Aut}}
\newcommand\TP{\mathrm{S}}
\newcommand\treeq{\trianglelefteqslant}
\renewcommand\O{\mathcal{O}}
\newcommand\tensor{\otimes}
\newcommand\germ[2]{[#2]_{#1}}
\newcommand\Vg{\Gamma}

\newcommand\alg[1]{#1^{\mathrm{a}}}
\newcommand\RV{\mathrm{RV}}
\newcommand\rv{\mathrm{rv}}
\newcommand\val{\v}
\renewcommand\v{\mathrm{v}}
\newcommand\Qq{\mathbb{Q}}
\newcommand\Zz{\mathbb{Z}}
\newcommand\cut{\mathrm{cut}}
\newcommand\A{\mathrm{A}}
\newcommand\B{\mathrm{B}}

\newcommand\GL{\mathrm{GL}}

\newcommand\Balls{\mathrm{B}}
\newcommand\gBalls{\mathrm{B_g}}

\newcommand{\red}{\operatorname{red}}
\newcommand{\Stab}{\operatorname{Stab}}
\newcommand{\tp}{\operatorname{tp}}
\renewcommand\epsilon{\varepsilon}
\newcommand\Lin{\mathrm{Lin}}
\newcommand\Tp{\TP}
\newcommand\I{\mathrm{I}}
\newcommand\Cc{\mathcal{C}}
\newcommand\Cut{\mathrm{Cut}}
\newcommand\pCut{\Cut^\star}
\newcommand\aCut{\Cut^{\star\star}}
\newcommand\Can{\Lambda}
\newcommand\Mod{\mathrm{Mod}}
\newcommand\isom{\simeq}
\newcommand\Ga{\mathbb{G}_\mathrm{a}}
\newcommand\Gm{\mathbb{G}_\mathrm{m}}
\newcommand\restr[2]{#1|_{#2}}
\newcommand\Aut{\mathrm{Aut}}
\newcommand\vect{\mathrm{vect}}
\newcommand\m{\mathfrak{m}}
\newcommand\MM{\mathbb{M}}

\newcommand\cT{\mathcal{T}}
\newcommand\ur[1]{#1^{\mathrm{ur}}}
\newcommand\res{\mathrm{res}}

\renewcommand\mid{:}

\newcommand\Uu{\mathrm{U}}
\newcommand\Dd{\mathrm{D}}

\newcommand\Lat{\mathrm{Gr}}
\newcommand\Tor{\mathrm{Lin}}
\newcommand\PP{\mathbb{P}}

\newcommand\dual[1]{#1^{\vee}}
\newcommand\Hom{\mathrm{Hom}}
\newcommand\Vv{\mathrm{V}}
\newcommand\lineq[1]{#1^{\mathrm{leq}}}

\newcommand\D{\textbf{\textup{D}}\xspace}

\renewcommand\k{\mathrm{k}}
\newcommand\NIP{\mathrm{NIP}}

\newcommand\Gal{\mathrm{Gal}}
\newcommand\Igp{I_\mathrm{gp}}
\newcommand\ac{\mathrm{ac}}

\renewbibmacro{in:}{}
\bibliography{biblio}

\title{Imaginaries in equicharacteristic zero henselian fields}
\author{Silvain Rideau-Kikuchi\thanks{The first author was partially supported
by GeoMod AAPG2019 (ANR-DFG), Geometric and Combinatorial Configurations in
Model Theory} \and Mariana Vicar\'ia}

\begin{document}

\maketitle

\begin{abstract}
We classify the imaginaries in a large class of equicharacteristic zero
henselian valued fields that contain all those with bounded inertia group, and
more. To do so, we consider a mix of sorts introduced in earlier works of the two
authors and prove elimination of imaginaries down to the field, the k-linear
imaginaries and the imaginaries of the value group.
\end{abstract}


\section{Introduction}

In the model theory of valued fields, one of the most striking results is a
theorem by Ax, Kochen and, independently, Ershov which states that the
first-order theory of an equicharacteristic zero henselian valued field is
completely determined by the first-order theories of its residue field \(\k\)
and of its value group \(\Vg\). A natural philosophy follows from this theorem: the
model theory of a henselian valued field is controlled by its residue field and
its value group.

In the early 2000s, Hrushovski asked if this philosophy also applies to the
elimination of imaginaries: the classification of interpretable sets (quotients
of definable sets by definable equivalence relations), or equivalently, the
description of (rough) moduli spaces for families of definable sets. He proposed
a classification of imaginaries reminiscent of the Ax-Kochen-Ershov principle.
The goal of the present paper is to establish this classification for a broad
class of henselian valued fields of equicharacteristic zero, including all those
with bounded Galois group, that is, having finitely many extensions of any given
degree. In fact, it suffices that the maximal unramified algebraic extension has
bounded Galois group, in other words, when the inertia group is bounded.

The study of imaginaries in various henselian valued fields has been ongoing for
the past 20 years, starting with the case of algebraically closed valued fields
(\(\ACVF\)) in the foundational work by Haskell, Hrushovski and Macpherson
\cite{HHM-EI}. They proved that in $\ACVF$, every quotient can be described as a
subset of products of certain specific quotients, known as the \emph{geometric
sorts} : the main field \(K\), and, for all $n \in \Zz_{>0}$, the space
$\Lat_{n}:= \GL_{n}(K)/\GL_{n}(\O)$ of free rank $n$ $\O$-submodules of $K^{n}$
(also known as the affine grassmanian of \(\GL_{n}\)) and the space $\Tor_{n}=
\coprod_{R \in \Lat_{n}} R/\m R$, where $\O$ denotes the valuation ring and
$\m\subseteq \O$ is the unique maximal ideal. We say that \(\ACVF\) eliminates
imaginaries down to the geometric sorts. These results were later extended to
other specific henselian fields, potentially with additional structure
\cite{Mel-RCVF,padics,separablyclosed,VDF}.

As can be expected, the residue field and the value group create natural
obstructions to the elimination of imaginaries in valued fields. In earlier
work, the authors studied the residual obstructions
\cite{HilRid-EIAKE} and the value group obstructions \cite{Vic-EIACk}
independently. The main focus of the present work, and where its novelty and
complexity lie, is to understand how they interact. This leads us to a
classification under no assumptions on the residue field and very mild
assumptions on the value group.

The classification of imaginaries in \(\ACVF\) laid the groundwork for a rich
"geometric model theory" of valued fields. However, the dependence on
algebraically closed fields hides most of the more arithmetic phenomena. It is
the authors' belief that the present classification of imaginaries and its
methods will lead to a geometric model theory of valued fields that accounts for
their arithmetic. This can be seen for example in their work \cite{CKRKVic} with
Cubides Kovacsics on residue domination.

\subsection{Obstructions arising from the value group}

Let \(K\) be an equicharacteristic zero henselian valued field. When describing
imaginaries in \(K\), the imaginaries of the value group itself have to be
accounted for. Moreover, the complexity of the value group also directly impacts
the complexity of definable \(\O\)-modules and this also needs to be taken into
account.

This can be done by introducing the \emph{stabilizer sorts} which provide codes
for all the definable $\O$-submodules of \(K^n\), for any \(n\). More precisely,
let \(\Cc = (\Cc_c)_{c\in\pCut}\) be the (ind-)interpretable family of definable
proper cuts in \(\Vg\) --- here a cut is said to be proper if it is neither of
the cuts at infinity. For every \(c\in\pCut\), let \(\I_c\) denote the
\(\O\)-submodule \(\{x\in K\mid v(x) \in \Cc_c\}\). For every tuple \(c\) in
\(\pCut\), let \(\Can_c\) be the module \(\sum_i I_{c_i}e_i\), where
\((e_i)_{i<|c|}\) is the canonical basis of \(K^{|c|}\).

The group \(\B_n\) of upper triangular matrices acts on the set of all definable
\(\O\)-submodules of \(K^n\). We define
\[\Mod_c = \B_n/\Stab(\Can_c)\] and we identify an element of \(\Mod_c\) with
the corresponding \(\O\)-submodule (see \cref{Mod}.2).

We also consider the (ind-)interpretable set
\[\Mod = \coprod_{c} \Mod_c\] where \(c\) varies over the (ind-)interpretable set of
finite tuples in \(\pCut\). Any definable $\O$-submodules of \(K^n\) is coded in
\(\Mod\cup K\) (see \cref{code mod}).

When the value group has bounded regular rank (\emph{i.e.} there are at most
countably many definable convex subgroups in any elementary
extension\footnote{We refer the reader to \cite{Vic-EIOAG} for details on
bounded regular rank groups.}) and the residue field is algebraically closed,
the second author \cite[Theorem~5.12]{Vic-EIACk} showed that, together with the
imaginaries of the value group, these are essentially the only new imaginaries.

\subsection{Obstructions arising from the residue field}

When the residue field is not algebraically closed, new obstructions arise. The
residue field itself can have non-trivial imaginaries, but it might also induce
linearly twisted imaginaries on interpretable \(\k\)-vector spaces.

Let \(R\subseteq K^n\) be a definable \(\O\)-submodule. Then the quotient $R/\m
R$ is a $\k$-vector space of dimension $d\leq n$, on which \(\k\) induces a
non-trivial structure. Once we name a basis, $R/\m R$ is definably isomorphic to
some $\k^{d}$, and hence imaginaries of $R/\m R$ can be identified with
imaginaries of \(\k\); but this identification is not canonical and depends on a
choice of basis.

The structure $(\k, R/\m R)$ can be seen as a structure in the language
$\L_{\vect}$ with two sorts:
\begin{itemize}
\item a sort for $\k$ with the structure induced by \(K\);
\item a vector space sort $\Vv$ with the (additive) group language;
\item A function $\lambda: \k \times \Vv \rightarrow \Vv$ interpreted as scalar
multiplication. 
\end{itemize}
Given a set $X$ interpretable without parameters in the $\L_{\vect}$-theory of
dimension $d$ vector spaces, the interpretable sets \(X^{(\k,R/\m R)}\) have to
be accounted for.\footnote{In \cite{HilRid-EIAKE}, it is claimed that it
suffices to consider interpretable sets of the form \(V/E\), this claim appears
to be incorrect.}

Let $\aCut=\pCut\setminus\{\Vg_{>\gamma}\mid \gamma\in\Gamma\}$ --- unless the
value group is discrete, in which case we set \(\aCut = \pCut\). A module
\(R\subseteq K^n\) is said to be \(\m\)-avoiding if it is in \(\Mod_c\), for
some tuple \(c\) in \(\aCut\). The dimension of \(R/\m R\) only depends on \(c\)
--- it is equal to \(d = |\{i \mid \Cc_{c_i} = \Vg_{\leq\gamma}\), for some
\(\gamma\in\Vg\}|\). For every quotient \(X\) as above, we define
\[\Tor_{c,X} =\coprod_{R \in \Mod_{c}} X^{(\k,R/\m R)}\]
and the \emph{(generalized) \(\k\)-linear imaginaries}:
\[\lineq{\k} = \coprod_{c,X} \Tor_{c,X},\] where \(c\) ranges over all finite
tuples in \(\aCut\). Among those, we denote \(\Lat = \coprod_{c} \Mod_c\) and
\(\Tor = \coprod_{c} \Tor_{c,\Vv}\). Along with \(K\), these form the
\emph{(generalized) geometric sorts}, and they encode all definable
\(\O\)-submodules of \(K^n\), for any \(n\) (see \cref{code mod G}).

Assuming that the value group is elementarily equivalent to \(\Qq\) or \(\Zz\),
Hils and the first author \cite[Theorem 6.1.1]{HilRid-EIAKE} show that, under
the technical assumption that the residue field eliminates \(\exists^\infty\),
all new imaginaries essentially arise in this manner.

\subsection{An imaginary Ax-Kochen-Ershov principle}

In this paper, we provide a common generalization of
\cite{HilRid-EIAKE,Vic-EIACk} by obtaining a general Ax-Kochen-Ershov principle
for the classification of imaginaries, under a mild technical assumption on the
value group (we refer the reader to \cref{imaginary} for notation and
definitions related to imaginaries). However, note that dealing with both an
arbitrary residue field and a very general value group introduces new issues
that were not present in either earlier works. We give details on some of the
new tools required to deal with these issues after stating the main theorems.

\begin{definition}
\label{def Delta type}    
Given a structure \(M\), tuples of variables \(x\) and \(y\) and a set of
formulas \(\Delta(x,y)\), a \(\Delta\)-type \(p(x)\) over \(M\) is a maximal
finitely consistent set of formulas of the form \(\phi(x,a)\) and
\(\neg\phi(x,a)\), where \(\phi\in\Delta\) and \(a\) is a \(y\)-tuple in \(M\).
If \(A\subseteq M\), we say that such a type is \(A\)-definable if for every
\(\phi\in \Delta\), the set \( \{a\in M^y \mid \phi(x,a)\in p\} \) is \(A\)-definable.
\end{definition}

\begin{definition}\label{propd} We say that an ordered group \(G\) (potentially
with additional structure) satisfies Property \D if for every \(A =
\acleq(A)\subseteq \eq{G}\) and every finite set of \(A\)-formulas $\Delta(x,y)$
containing the formula $x< y_{0}$, any $A$-definable \(\Delta\)-type $p(x)$ over
\(G\) is contained in an $A$-definable complete type \(q(x)\) over \(G\).
\end{definition}

This is a stronger property than the density of definable types. It holds in
ordered abelian groups of bounded regular rank with no addition structure (see
the second half of the proof of \cite[Theorem~5.3]{Vic-EIACk}).

Our main results are the following. Let \(K\) be an equicharacteristic zero
henselian valued field such that the value group is \emph{either}:
\begin{itemize}
\item dense with property \D;
\item or, a discrete ordered abelian group of bounded regular rank --- in which
case, we add a constant for a uniformizer.
\end{itemize}

\begin{theorem*}[\cref{ac}]
Let \(K_\ac\) be an expansion of \(K\) by angular components. Then $K_\ac$
weakly eliminates imaginaries down to $K \cup \lineq{\k} \cup \eq{\Vg}$ --- in
other words, any interpretable set admits an interpretable finite cover by a
subset of some cartesian power of $K \cup \lineq{\k} \cup \eq{\Vg}$ (see
\cref{imaginary} for precise definitions).
\end{theorem*}

Without an angular component the short exact sequence
\[1 \to \k^\star \to \RV = K^\star/(1+\m) \to \Vg \to 0\] might not eliminate
imaginaries down to \(\lineq{\k}\) and \(\eq{\Gamma}\), creating further
obstructions (see \cref{general RV im}). They can be avoided by further
assumption on \(\Gamma\):

\begin{theorem*}[\cref{AKE}]
Assume that for every $n \in \Zz_{\geq 1}$ we have $[\Vg: n\Vg]< \infty$, and we add
constants in \(\RV\) so that \(\Vg/n\Vg = \Vg/n v(\RV(\acl(\emptyset)))\).
Then $K$ weakly eliminates imaginaries down to $K \cup \lineq{\k}\cup \eq{\Vg}$. 
\end{theorem*}

\begin{remark}
These theorems remain true when \(K\) comes with additional structure on \(\k\)
and, independently \(\Gamma\) --- the main exception is that when the valuation
is discrete, we require the valuation group to have no additional structure.
Also, in the second theorem, the assumption on the finiteness of
\(\Gamma/n\Gamma\) can be replaced by asking that \(\k\) is an algebraically
closed field with no additional structure (see \cref{wei res ACF}, this is
essentially \cite[Theorem 5.12]{Vic-EIACk}).
\end{remark}

These results generalize all previously known weak elimination results in
equicharacteristic zero (in particular, those of \cite{HilRid-EIAKE,Vic-EIACk})
and provides a definitive answer for, among others, all equicharacteristic zero
henselian valued fields with bounded inertia group  --- meaning that the maximal
unramified extension has bounded Galois group.

As a corollary (and an illustration) of these two results, when \(\Vg\) is
dense, we give a complete classification of (almost) \(\k\)-internal sets ---
see \cref{kint}. In the case of \(\ACVF\) this classification is a cornerstone
of the study of stable domination and the subsequent work of Hrushovski and
Loeser \cite{HruLoe} on Berkovich spaces. These new results are used in
\cite{CKRKV-resdom} to study residual domination in equicharacteristic zero
henselian fields.

\subsection{Sketch of the proof}

The proof of the main theorems proceeds in two main steps. The first one is a
density result for definable types in the structure induced by the maximal
unramified algebraic extension (\cref{density}). Density of definable (complete)
types is at the core of many recent results on elimination of imaginaries.
However, it cannot hold in arbitrary equicharacteristic zero henselian field
\(K\) because it might not hold in the residue field or the value group.

In \cite[Theorem~3.1.3]{HilRid-EIAKE} Hils and the first author showed that
definable types of the algebraic closure $\alg{K}$ are dense among definable
sets in $K$, under the assumption that the residue field eliminates
$\exists^{\infty}$. The second author (\cite[Theorem~5.9]{Vic-EIACk}) proved
density of definable types in the maximal unramified algebraic extension
\(\ur{K}\), assuming the value group has bounded regular rank. \cref{density}
generalizes both results and unifies them by proving that the definable types of
\(\ur{K}\) are dense among the sets definable in \(K\) (assuming Property \D).
Note also that no hypothesis on the residue field is required anymore. A
significant new challenge in this construction is to relate the germs of
functions definable in \(K\) to those of functions definable in \(\ur{K}\) ---
see \cref{germs dens}.

The second step of proof consists in studying the completions of the partial
definable types constructed in the first step. We show that they are invariant
over \(\RV\) and \(\k\)-vector spaces of the form \(R/ \m R\), where \(R\) is a
definable \(\O\)-submodules of \(K^n\) (\cref{invariantcompletions}). The bulk
of the work (\cref{inv res}) revolves around showing that (generalized)
geometric points can be lifted to the valued field by a sufficiently invariant
type. This, in turn, relies heavily on the technical computation of germs of
function taking values in sets of the form \(R/ \m R\).

We describe such germs in three steps. First we first consider the case of
valued fields with algebraically closed residue field (\emph{cf.} \cref{germs
intk}). We then consider valued fields with dense value group and arbitrary
residue field (\cref{subsec: dense}). This relies on the characterization of the
$\k$-internal sets (\cref{almost k int}). Lastly, we consider valued fields with
discrete value group and arbitrary residue field (\cref{subsec: discrete}). In
that case, we circumvent the characterization of the $\k$-internal sets by
considering a ramified extension with dense value group.

Together these two results imply a first elimination result relative to the
imaginaries of \(\RV\) and the \(\k\)-vector spaces of the form \(R/ \m R\)
(\cref{weakcode}). The main theorems follow by describing the imaginaries in
\(\RV\) with and without angular component.

\subsection{Overview of the paper}

\cref{sec: prel} provides some preliminary reminders on imaginaries and the
model theory of equicharacteristic zero henselian fields.

In \cref{sec:code mod}, we introduce the stabilizer sorts and the generalized
geometric sorts. Both provide codes for the definable $\mathcal{O}$-modules. We
prove a unary decomposition for the stabilizer sorts (\cref{decomp solv}). The
generalized geometric sorts --- and the related notion of \(\m\)-avoiding module
--- play a crucial role in classifying \(\k\)-internal sets among the stabilizer
sorts (\cref{almost k int}), provided that the value group is dense.

In \cref{densitysec}, we prove the density of definable types of the maximal
unramified algebraic extension \(\ur{K}\) among definable sets of \(K\)
(\cref{density}). In \cref{invext}, we show that those definable types have
invariant completions (\cref{invariantcompletions}). Finally, in
\cref{conclusion}, we wrap everything together and show the two main theorems. 

\subsection*{Acknowledgments}

The authors are ever grateful to M. Hils whose decade long collaboration with
the first author was foundational to the present paper. They would also like to
thank E. Hrushovski, T. Scanlon and P. Simon for many enlightening discussions
on this topic. Finally, they would like to thank anonymous referees for their
many insightful comments on earlier versions of this paper.

\section{Preliminaries}\label{sec: prel}

\subsection{Model theoretic preliminaries}
\label{imaginary}

Throughout this text, if \(M\) is an \(\L\)-structure, a set \(X\) is said to be
"definable" if it is definable with parameters. If we specify that it is
"\(\L\)-definable" we mean that it is definable without parameters. Also, we
extend definable sets canonically to elementary extensions of \(M\) and we
distinguish the definable set \(X\) from the set \(X(M)\) of its realizations in
\(M\).

If \(A \subseteq M\), we write \(X(A)\) for the points of \(X\) whose
coordinates are in \(A\), rather than \(X\cap\dcl(A)\). We change language too
often to not be explicit with the definable closures at play.

We refer the reader to \cite[Section~8.4]{TenZie} for a detailed exposition of
the elimination of imaginaries. Let \(T\) be an \(\L\)-theory. Consider the
language \(\eq{\L}\) obtained by adding to \(\L\) a new sort \(S_X\) for every
\(\L\)-definable set \(X \subseteq Y\times Z\), where \(Y\) and \(Z\) are
product of sorts, and a new symbol \(f_X : Z \to S_X\). The \(\eq{\L}\)-theory
\(\eq{T}\) is then obtained as the union of \(T\), the fact that the \(f_X\) are
surjective and that their fibers are the classes of the equivalence relation
defined by \(X_{z_1} = \{y\in Y: (y,z_1)\in X\} = X_{z_2}\).

Any \(M\models T\) has a unique expansion to a model of \(\eq{T}\) denoted
\(\eq{M}\) --- whose points are called the \emph{imaginaries}. A set is
interpretable if it is definable in \(\eq{M}\). Throughout this paper, when
considering types, definable closures or algebraic closures, we will work in the
\(\eq{\L}\)-structure, unless otherwise specified.

Given \(M\models T\) and an \(\L(M)\)-definable set \(X\), we denote by
\(\code{X} \subseteq \eq{M}\) the intersection of all \(A = \dcleq(A)\subseteq
\eq{M}\) such that \(X\) is \(\eq{\L}(A)\)-definable. It is the smallest
\(\dcleq\)-closed set of definition for \(X\). Any \(\dcleq\)-generating subset
of \(\code{X}\) is called \emph{a code} of \(X\). More generally, if
\(A\subseteq\eq{M}\) is a set of parameters, any tuple \(e\) such that
\(\dcl(Ae) = \dcl(A\code{X})\) is called a code of \(X\) over \(A\).

If \(\mathcal{D}\) is a collection of sorts of \(\eq{\L}\) (equivalently, a
collection of \(\L\)-interpretable sets) and \(A\subseteq \eq{M}\) is a set of
parameters, we say that \(X\) is \emph{coded} in \(\mathcal{D}\) over \(A\) if
it is \(\eq{\L}(A\cup\mathcal{D}(\code{X}))\)-definable --- \emph{i.e.}, it
admits a code in \(\mathcal{D}\) over \(A\).

The theory \(T\) is said to \emph{eliminate imaginaries} down to \(\mathcal{D}\)
if, for every \(M\models T\), every \(\L(M)\)-definable set \(X\) is coded in
\(\mathcal{D}\) --- equivalently, for every \(e\in\eq{M}\), there is some \(d\in
\mathcal{D}(\dcleq(e))\) such that \(e\in\dcleq(d)\). Finally, we say that the
theory \(T\) \emph{weakly eliminates imaginaries} down to \(\mathcal{D}\) if for
every \(e\in\eq{M}\), there is some \(d\in \mathcal{D}(\acleq(e))\) such that
\(e\in\dcleq(d)\).

If \(p(x)\) is a (definable) partial type over some structure \(M\) and \(f\)
and \(g\) are definable functions in \(M\) which are defined at realizations of
\(p\), we say that they have the same \(p\)-germ, and we write \(\germ{p}{f} =
\germ{p}{g}\) if \(p(x)\vdash f(x) = g(x)\). When \(p\) is definable we write
\(\germ{p}{f}\) for class of \(p\)-germs as \(f\) varies in an
\(\L\)-definable family.

Lastly, we refer the reader to \cite[Appendix $\A$ (Definition $\A.2$ and
$\A.5$)]{Rideau} for a detailed presentation of expansions (called enrichments,
there) and relative quantifier elimination. Note that in the present text,
expansions do not allow adding new sorts.

\subsection{Equicharacteristic zero henselian fields}\label{prel}

Let \(\Hen[0,0]\) be the theory of equicharacteristic zero henselian valued
fields \((\K,\v)\) with no additional structure, in some language \(\L\). The
exact language we use does not matter much since we will be working \(\eq{\L}\).
In this section, we recall some useful results about these structures. We denote
by \(\RV^\star\) the group \(\K^\star/(1+\m)\), where \(\m\) is the maximal
ideal of the valuation ring \(\O\subseteq\K\) and \(\rv : \K\to\RV =
\RV^\star\cup\{0\}\) the canonical projection (extended by \(\rv(0) = 0\)).

\begin{theorem}[{\cite[Theorem~B]{Bas}}]
\label{EQ RV}
Let \(M\models\Hen[0,0]\) and \(A\leq\K(M)\) be a subring. Every \(A\)-definable
subset of \(\K^x\times\RV^y\) is of the form \(\{(x,y) \mid (\rv(P(x)),y)\in
X\}\), for some tuple \(P\) in \(A[x]\) and some \(X\subseteq \RV^n\) which is
\(\rv(A)\)-definable in the short exact sequence
\[1\to\k^\star\to\RV^\star\to \Vg \to 0.\] where \(\k = \O/\m\) is the residue
field and \(\Vg =\v(\K)\) is the value group.
\end{theorem}

This remains true in \(\RV\)-expansions --- \emph{i.e.} when \(\RV\) comes with
additional structure.

From the result above, either by adding a section or proving a quantifier
elimination result for short exact sequences, we can deduce the following:

\begin{proposition}
\label{Vg k structure}
Let \(M\models\Hen[0,0]\) and \(A\leq\K(M)\) be a subring. The sets \(\k\) and
\(\Vg\) are stably embedded (with respectively the structure of a field and an
ordered group) and they are orthogonal. In other words, any \(M\)-definable
subset of \(\k^x\times\Vg^y\) is a finite union of products \(X\times Y\) where
\(X\) is definable in the field \(\k\) and \(Y\) is definable in the ordered
group \(\Vg\).

Moreover, any \(A\)-definable \(X\subseteq \Vg^n\) is
\(\v(A)\)-definable --- we say that \(\Vg\) is \emph{strongly stably embedded}.
\end{proposition}

The first part is \cite[corollary 2.25]{vdD-notes}. The second part follows
\cite[corollary 2.24]{vdD-notes}, noting that adding a angular component does
not grow the generated structure in \(\Vg\).

These results remain true in \(\k\)-\(\Vg\)-expansions --- \emph{i.e.} when
\(\k\) comes with additional structure, and, independently so does \(\Vg\). They
also hold if an angular component is added to the language.

\begin{lemma}
\label{acl Vg}
Let \(M\) be an expansion of a model of \(\Hen[0,0]\) with strongly stably
embedded value group \(\Vg\). For any subring \(A\leq \K(M)\), we have
\[\eq{\Vg}(\acl(A)) \subseteq \acl(\v(A)).\]
\end{lemma}

\begin{proof}
Let \(X\subseteq \Vg^n\) be \(\acl(A)\)-definable. Let \(X_1 = X,\ldots,X_n\) be
the \(\aut(M/A)\)-conjugates of \(X\) over \(A\). For every \(x,y\in \Vg^n\), we define
\(x E y\) to hold if, for every \(i\), \(x\in X_i\) if and only if \(y\in X_i\).
This is an \(A\)-definable equivalence relation. By strong stable embeddedness,
it is \(\v(A)\)-definable. Note also that \(E\) has finitely many equivalence
classes and \(X\) is a union of classes. It follows that \(X\) is
\(\acl(\v(A))\)-definable.
\end{proof}

\begin{definition}
\label{def gen ball}
Let \(M\models\Hen[0,0]\). Let \(a\in\K(M)\) and let  \(C\) be a cut in
\(\Vg(M)\) --- that is, an upwards closed subset. We define the generalized ball
\(b_C(a)\) of cut \(C\) around \(a\) to be \(\{x\in\K\mid \v(x-a) \in C\}\). A
generalized ball is said to be open if its cut is not of the form \(\Vg_{\geq
\gamma}\), for each \(\gamma\in\Vg(M)\).

Let \(\gBalls\) denote the (ind-)interpretable set of (codes for) definable
generalized balls.
\end{definition}

Note that, for every \(\gamma\in\Vg(M)\),  \(b_{\Vg_{>\gamma}}(a)\) is the open
ball of radius \(\gamma\) around \(a\), \(b_{\Vg_{\geq\gamma}}(a)\) is the
closed ball of radius \(\gamma\) around \(a\) --- and \(b_{\Vg}(a) = \K\) is
also considered an open ball.

If \(\Gamma(M)\) is discrete, \(b_{\Vg_{>\gamma}}(a) = b_{\Vg_{\geq\gamma}}(a)\)
is not considered to be an open generalized ball. Let \(M\) be an
\(\RV\)-expansion of a model of \(\Hen[0,0]\) and let \(A\leq \K(M)\) be a
subring. Let \(\alg{M}\) denote the algebraic closure as a pure valued field.

\Cref{EQ RV} can also be refined for unary sets:

\begin{proposition}[{\cite[Proposition~3.6]{Fle}}]
\label{Hen prep}
Let \(M\models\Hen[0,0]\), let \(A\leq\K(M)\) be a subring and let \(X\subseteq \K
\times \RV^n\) be \(A\)-definable. There exists a finite set \(C\subseteq\alg{A}
\cap \K(M)\) such that for every \(\xi\in\RV^n\), \(X_\xi = \{x\in\K\mid
(x,\xi)\in X\} = \rv_C^{-1}(\rv_C(X_\xi))\) where \(\rv_C(x) = (\rv(x-c))_{c\in
C}\).
\end{proposition}

It follows that for any generalized ball \(b\) that does not intersect \(C\),
either \(b\cap X_\xi = b\) or \(b\cap X_\xi = \emptyset\). We say that \(C\)
prepares \(X\).

This proposition remains true in \(\RV\)-expansions. Also, it follows that
\(\K(\acl(A)) \subseteq \alg{A}\).

\begin{lemma}
\label{points balls}
Let \(M\) be (an \(\RV\)-expansion of) a model of \(\Hen[0,0]\), let
\(A\leq\K(M)\) and let \(b\) be an \(\acl(A)\)-definable generalized ball which
is not an open ball. Then there exists \(c\in \alg{A}\cap b(M)\) whose other
\(\aut(M/A)\)-conjugates are all outside of \(b\).

Moreover, if \(b\) is a ball, we may assume that no other
\(\aut(\alg{M}/A)\)-conjugate is in \(b(\alg{M})\).
\end{lemma}

\begin{proof}
Note that any other \(\aut(M/A)\)-conjugate of \(b\) is disjoint from \(b\). Let
\(B\) be the union of \(\aut(M/A)\)-conjugates of \(b\). It is an
\(A\)-definable set. By \cref{Hen prep}, there exists \(C\subseteq \alg{A}\cap
\K(M)\) such that for any ball \(d\) disjoint from \(C\), either \(d\subseteq
B\) or \(d\cap B = \emptyset\). If \(b\cap C = \emptyset\), let \(d\) be the
largest ball containing \(b\) which is disjoint from \(C\) --- \emph{i.e.} the
open ball around \(b\) with radius \(\min_{c\in C}\v(x-c)\), for any \(x\in b\).
Since \(d\) contains \(b\) and is disjoint from \(C\), it is contained in \(B\).
So \(d\) is covered by finitely many disjoint subballs. As the residue field is
infinite, this is impossible unless \(d = b\), in which case \(b\) would be
open, contradicting our assumption.

So \(b\cap C \neq \emptyset\). Let \(C_b = (\aut(M/A)\cdot C)\cap b\). Since
we are in equicharacteristic zero, the average \(c\) of \(C_b\) is in \(b(M)\).
By construction, no other \(\aut(M/A)\)-conjugate of \(c\) is in \(b\).

If \(b\) is a ball, then it is definable in \(\alg{M}\) and its
\(\aut(\alg{M}/A)\)-orbit is also a finite set of disjoint balls. Let  \(C'_b =
\aut(\alg{M}/A)\cdot c\). Since \(M\) is henselian, the average \(c'\) of
\(C'_b\) is in \(M\). By construction again, the only \(\aut(\alg{M}/A)\)-conjugate of
\(c'\) in \(b\) is \(c'\).
\end{proof}

Finally, when the residue field is algebraically closed, \cref{EQ RV} can be further simplified:

\begin{theorem}[{\cite[Corollary~2.33]{Vic-EIACk}}]
\label{EQ Vg}
Assume the residue field \(\k(M)\) is algebraically closed. Every
\(A\)-definable subset of \(\K^x\) is of the form \(\v(P(x))\in X\) where \(P\)
is a tuple in \(A[x]\) and \(X\subseteq \Vg^n\) is
\(\v(A)\)-definable in the ordered group structure. Moreover, this remains true
in \(\Vg\)-expansions --- \emph{i.e.} when \(\Vg\) comes with additional
structure.
\end{theorem}

\section{Codes of \texorpdfstring{\(\O\)}{O}-modules}\label{sec:code mod}

\subsection{The stabilizer sorts}\label{stabsorts}
Let \(M\) be an (expansion of a) valued field in some language \(\L\).

\begin{notation}
We fix an (ind-)\(\L\)-interpretable family \(\Cc = (\Cc_c)_{c\in\Cut}\) of
cuts in \(\Vg\) such that any \(M\)-definable cut is of the of form \(\Cc_c\)
for some unique \(c \in \Cut(M)\). We will further assume that \(c\) is a
canonical parameter for \(\Cc_c\).
\end{notation}

For every \(c\in\Cut\), let \(\I_c\) denote the \(\O\)-submodule \(\{x\in \K\mid
\val(x) \in \Cc_c\}\). Note that, by definition of \(\Cc\), any
\(\L(M)\)-definable \(\O\)-submodule of \(\K\) is of the form \(\I_c\) for some
unique \(c\in \Cut(M)\). We also denote \(\Delta_c = \{\gamma\in\Vg \mid \gamma
+ \Cc_{c} = \Cc_{c}\}\) --- it is a convex subgroup of \(\Vg\).

The following results are well-established and go back to Bauer's work on
separated extensions. 

\begin{definition}
A definable valuation \(v\) on an interpretable \(\K\)-vector space \(V\) is a
map to some interpretable set \(X\) with an order preserving action of \(\Vg\)
such that
\begin{itemize}
\item for every \(\lambda\in \K\) and \(x\in V\), \(v(\lambda x) = \v(\lambda) +
v(x)\);
\item for every \(x,y\in V\), \(v(x+y) \geq \min \{v(x),v(y)\}\).
\end{itemize}
\end{definition}

Note that we do not assume that \(v(x) = v(0)\) implies \(x = 0\), nor do we
assume that the action of \(\Gamma\) on \(X\) is free.

\begin{proposition}
\label{uppertrian}
Assume that \(M\) is definably spherically complete --- that is, the
intersection of any \(M\)-definable chain of balls is non empty.
\begin{enumerate}
\item For every \(M\)-definable valuation \(v\) on \(\K^n\), there
exists a triangular basis \((a_i)_{i<n}\) of \(\K^n\) such that, for all \(i\),
\(v(a_i)\in \dcleq(\code{v})\) and for every \(\lambda_i\in\K\),
\[v(\sum_i \lambda_i a_i) = \min_i \v(\lambda_i)\cdot v(a_i).\]

\item Any \(M\)-definable \(\O\)-submodule \(R\) of \(\K^n\) is of the form
\(\sum_{i< n} \I_{c_i} a_i\), where \(a_i\) is a triangular basis of \(\K^n(M)\)
and \(c_i\in\Cut(\code{R})\).
\end{enumerate}
\end{proposition}

A basis as in the fist assertion is said to be separated. A module as in the
second assertion is said to be of type \(c = (c_i)_{i<n}\).

\begin{proof}
If \(M\) is (elementarily equivalent to a) maximally complete field, the first
assertion is \cite[Lemma~5.7]{Vic-EIACk}. If \(M\) is only definably spherically
complete, the same proof works using \cite[Claim~3.3.9]{HilRid-EIAKE} instead of
\cite[Fact~2.55]{Vic-EIACk}.

Let us now prove the second assertion. For every \(x\in \K^n\), we define
\(v_R(x) = \{\v(\lambda) \mid \lambda x \in R\}\) a (non-empty) cut of \(\Vg\).
We order them by inclusion (so \(\Vg = v_R(0)\) is the maximal element and
\(\{\infty\}\) is the minimal element). Note that, for every \(x\in \K\),
\(v_R(\lambda x) = -\v(\lambda) + v_R(x)\) and for this action of \(\Vg\) on the
set of cuts, \(v_R\) is an \(M\)-definable valuation.

By the first assertion, we can find a separated triangular basis \((a_i)_{i}\)
of \(\K^n(M)\), such that $v_{R}(a_{i}) \in \dcleq(\ulcorner R\urcorner)$. Then
\(\sum_i \lambda_i a_i \in R\) if and only if \(0 = \v(1) \in v_R(\sum_i
\lambda_i a_i) = \min_i - \v(x_i) + v_R(a_i)\), \emph{i.e.} \(\v(\lambda_i) \in
v_R(a_i)\) for all \(i\). Let \(c_i\in\Cut(\dcleq({\code{R}}))\) be such that
\(v_R(a_i) = \Cc_{c_i}\). We then have \(R = \sum_i \I_{c_i} a_i\), as required. 
\end{proof}

\begin{notation}
\label{Mod}
\begin{enumerate}
\item We write $B_{n}$ to denote the set of $n \times n$ upper triangular and
invertible matrices. We write $\Dd_{n}\leq \B_n$ for the subgroup of diagonal
matrices and \(\Uu_n\leq \B_n\) for the subgroup of unipotent matrices, that is
upper triangular matrices with ones on the diagonal.

\item For every \(n\)-tuple \(c\) in \(\Cut\), we define \(\Mod_c\) to be the
interpretable set of modules of type \(c\). Let \(\Can_c = \sum \I_{c_i} e_i\)
be the canonical module of type \(c\), where \(e_i\) is the canonical basis of
\(K^n\). Then \(\sum_{i< n} \I_{c_i} a_i = A \cdot \Can_c\) where \(A \in \B_n\)
is the upper triangular matrix of the \(a_i\). In other words, \(\B_n\) acts
transitively on \(\Mod_c\) and
\[\Mod_c \isom \B_n / \Stab(\Can_c).\] We will now identify \(\Mod_c\) with this
quotient of \(\B_n\) and for every \(s\in\Mod_c\), we write \(R_s\) for the
\(\O\)-module of type \(c\) coded by \(s\). Let \(\mu_c : \B_n\to \Mod_c\)
denote the natural quotient map.
\end{enumerate}
\end{notation}

If \(\Delta\leq\Vg\) is a (definable) convex subgroup, we write \(\O_{\Delta} =
\{x\in\K\mid \exists\delta\in\Delta\v(x)\geq \delta\}\) for the associated
(definable) valuation ring. If \(I,J \leq \K\) are two (definable)
\(\O\)-submodules, let \((I:J)\) denote the (definable) \(\O\)-submodule
\(\{x\in\K \mid xJ \subseteq I\}\).

\begin{proposition}
\label{descr stabilizer}
Let \(c\) be a tuple in \(\Cut\). For every \(a\in \B_n\), we have
\[a\in \Stab(\Can_c)\text{ if and only if }\left\{
\begin{array}{cl}
a_{i,i} \in \O_{\Delta_{c_i}}^\times&\text{for all }i<n\text{, and}\\
a_{i,j} \in (\I_{c_i}: \I_{c_j}) &\text{for all i < j < n.}
\end{array}
\right.\]
\end{proposition}

\begin{proof}
We proceed by induction on \(n\). Write \(a\) as
\(\left(\begin{smallmatrix}a_{0,0}&b\\ 0& e\end{smallmatrix}\right)\), with
\(e\in \B_{n-1}\), and \(c\) as \((c_0,d)\), with \(d\in \Cut^{n-1}\). If \(a
\Can_c \subseteq \Can_c\), then, considering the action on \(\I_{c_{0}}\) and
\(\Can_{d}\), we see that \(a_{0,0} \I_{c_0} \subseteq \I_{c_0}\), \(b
\Can_{d}\subseteq \I_{c_0}\) --- so, considering the action on each
\(\I_{c_{j}}\), for every \(j > 0\), \(a_{0,j} \I_{c_j} \subseteq \I_{c_0}\) ---
and \(e \Can_{d}\subseteq \Can_{d}\); and the converse also holds.

Since \(a \Can_c = \Can_c\) if, moreover, \(a^{-1} \Can_c =
\left(\begin{smallmatrix}a_{0,0}^{-1}&-a_{0,0}^{-1} be^{-1}\\
0& e^{-1}\end{smallmatrix}\right)\Can_c \subseteq \Can_c\), it follows that we
must further have \(a_{0,0} \I_{c_0} = \I_{c_0}\), \emph{i.e.} \(\v(a_{0,0}) \in
\Delta_{c_0}\) and \(e \Can_{d} = \Can_{d}\). These conditions are sufficient
since, in that case, \(a_{0,0}^{-1}be^{-1}\Can_{d} = a_{0,0}^{-1}b\Can_d
\subseteq a_{0,0}^{-1} \I_{c_0} =  \I_{c_0}\). The claim now follows by
induction.
\end{proof}

\begin{definition}
Let \(\pCut = \Cut\setminus \{\emptyset,\Vg\}\) and \(\Mod\) be the disjoint
union of all the \(\Mod_c\) where \(c\) is a tuple in \(\pCut\).
\end{definition}

Note that any \(s\in\Mod\) determines the unique \(c\) such that \(s\in
\Mod_c\).

\begin{corollary}
\label{code mod}
Any \(M\)-definable \(\O\)-submodule \(R\) of \(\K^n\) is coded in \(\K\cup\Mod\).
\end{corollary}

\begin{proof}
Let \(V\subseteq \K^n\) be the \(\K\)-span of \(R\) and \(W = \{x\in\K^n\mid \K
x \subseteq R\}\). Then, by \cite[Lemma~4.3]{Joh-EIACVF}, \(V/W\) is
\(\K(\code{R})\)-definably isomorphic to some \(\K^r\) and \(R\) is entirely
determined by its image in \(V/W\). So we may assume \(V = \K^r\) and \(W = 0\)
and hence that \(R\) is of type \(c\) with \(c\in \pCut\). By definition, it is
coded in \(\Mod_c\).
\end{proof}

\begin{remark}
There is a lot of redundancy in \(\Mod\). If \(c\) and \(c'\) are tuples in
\(\Cut\) of the same length such that for every \(i < n\), \(c_i'\) is a
translate of \(c_i\), then there is a natural bijection between \(\Mod_c\) and
\(\Mod_{c'}\) given by the action of a diagonal matrix.

If there exists an (ind-)interpretable subset \(\Cut'\subseteq\Cut\) such that any
definable cut is of the form \(a + \Cc_c\) for a unique \(c\in \Cut'\), it
follows that every \(M\)-definable \(\O\)-submodule of \(\K^n\)  is coded in
\(\K\cup\bigcup_{c\in \Cut'\setminus\{\emptyset,\Vg\}} \Mod_c\). Similarly, we
can replace \(\Cut\) by \(\Cut'\) in the definition of the geometric sorts
(\cref{geom sorts}).

This is the case, for example, in ordered abelian groups of bounded regular rank
(\emph{cf}. \cite[Corollary~2.24]{Vic-EIACk}).
\end{remark}

\begin{remark}
\label{code gballs}
Any \(M\)-definable generalized ball is inter-definable with the
sub-\(\O\)-module \(R\) of \(\K^2\) generated by \(b\times\{1\}\); indeed \(b =
\{x\in\K\mid (x,1)\in R\}\). So generalized balls are coded in \(\Mod\).
\end{remark}

Let us now describe the structure of \(\Mod\). The solvability of the upper
triangular invertible matrices will play a central role in this description.

We go through the elements of an upper triangular matrix diagonal by diagonal
starting with the middle diagonal, and in each diagonal, we proceed from top to
bottom. In other words, we order pairs \((i,j)\) such that \(i\leq j < n\) first
by \(j - i\) and then by \(i\). We will identify the set of such pairs with the
set of non-negative integers \(< n(n+1)/2\), according to that order.

For every pair \((i,j)\), let \(p_{i,j} : \B_n \to \K\) be the projection on
coordinate \((i,j)\). Let also \(\epsilon_{i,j} = 1\) if \(i = j\) and \(0\)
otherwise. For every pair \(\ell\), let \(G_{\ell} = \{a \in \B_n \mid p_{k}(a)
= \epsilon_k,\ \forall k < \ell\}\). Then \(G_0 = \B_n\) and \(G_{n(n+1)/2} =
\{\mathrm{id}\}\). By choice of the order, for every \(\ell\), \(G_{\ell+1}
\triangleleft G_{\ell}\) and \(p_\ell\) induces an isomorphism from \(G_{\ell} /
G_{\ell+1}\) to \(\Gm\), if \(\ell < n\), and to \(\Ga\) otherwise. Note also
that \(H_\ell =\{a\in G_\ell \mid p_k(a) = \epsilon_k,\ \forall k > \ell\}\), is
a section of \(p_\ell\) restricted to \(G_\ell\) and hence \(G_\ell =
G_{\ell+1}\rtimes H_\ell\).

Furthermore, we have \(G_n = \Uu_n\), \(\B_n = \Uu_n\rtimes \Dd_n\) and, for
every \(\ell \geq n\), \(G_\ell\) is central in \(G_n\) module \(G_{\ell+1}\)
--- actually modulo the next upper triangular group \(G_{0,j}\),--- if \(\ell\) is
a pair \((i,i+j-1)\). In particular \(G_\ell \trianglelefteq \Uu_n\).

We can now prove the following unary decomposition.

\begin{proposition}
\label{decomp solv}
Let \(s \in \Mod_c\). There exists a finite tuple \(b = (b_\ell)_{\ell <
n(n+1)/2}\) in \(\eq{M}\) --- we identify each \(b_\ell\) with a subset of some
\(\K^{r_\ell}\) --- and \(c b_{<\ell}\)-interpretable sets \(X_\ell\) such that:
\begin{itemize}
\item for every \(\ell\), \(b_\ell\in X_\ell\);
\item \(\dcleq(s) = \dcleq(cb)\);
\item if \(\ell < n\), then \(X_\ell = \Vg/\Delta_{c_i}\), where \(\ell = (i,i)\);
\item if \(\ell \geq n\), then \(X_\ell\) has a \(cb_{<\ell}\)-definable
\(\K/\I_\ell\)-torsor structure where \(\I_\ell\) is a \(cb_{<n}\)-definable
multiple of \((\I_{c_i}:\I_{c_j})\) and \(\ell = (i,j)\).
\end{itemize}

Moreover, for any  \(n\leq \ell < n(n+1)/2\) and any choice of \(a_{k} \in
b_{k}\), for \(k < \ell\), there is a (uniformly) \(c a_{<\ell}\)-definable
isomorphism of torsors \(f_\ell : X_\ell \to \K/\I_\ell\) and a \(c
a_{<\ell}\)-definable function \(g_\ell : f_\ell(b_\ell) \to b_\ell\).
\end{proposition}

\begin{proof}
Let \(F = \Stab(\Lambda_{c})\), we identify \(s\) with a coset \(gF\) for some
\(g\in\B_n\). Let \(d\in \Dd_n\) and \(u\in\Uu_n\) be such that \(g = ud\). Note
that the map \(B_n \to B_n/U_n \simeq D_n\) is the projection on the diagonal
coordinates. Therefore, by \cref{descr stabilizer}, the image of \(F\) in
\(B_n/U_n\) is identified with \(F_\Dd = F\cap \Dd_n\). Writing \(F_\Uu = F\cap
\Uu_n\), we thus have \(F = F_\Uu \rtimes F_\Dd\). By (the proof of)
\cite[Lemma~11.10]{HHM}, \(\dcleq(s) = \dcleq(c,\code{d
F_\Dd},\code{uF_\Uu^d})\), where \(F_\Uu^d = dF_U d^{-1} = dFd^{-1}\cap \Uu_n\)
does not depend on the choice of \(d\) in \(d F_D\). Then, by \cref{descr
stabilizer}, \(\Dd_n/F_\Dd \isom \prod_{\ell<n} \K^\star/\O^\star_{\Delta_\ell}
\isom \prod_{\ell<n} \Vg/\Delta_i\) and, for every \(\ell < n\), we chose
\(b_\ell = \val_{\Delta_\ell}(d_\ell)\) to be the \(\ell\)-th coordinate in this
product. Note that \(\dcl(c\code{d F_D}) = \dcl(c b_{<n})\).

Now, for every \(\ell \geq n\), note that \(F_\Uu^d G_\ell = G_\ell F_\Uu^d\) is
a subgroup, since \(G_\ell\) is normal in \(\Uu_n\) and moreover, \(F_\Uu^d
G_{\ell+1} \trianglelefteq F_\Uu^d G_\ell\) since for every \(g\in G_{\ell}\),
\((F_\Uu^d)^g \subseteq F_\Uu^d G_{\ell+1}\) by centrality of the sequence. For
every \(\ell\geq n\), let \(X_\ell = u F_\Uu^d G_\ell / F_\Uu^d G_{\ell+1}\) for
the right regular action and \(b_\ell = u F_\Uu^d G_{\ell+1} \in X_\ell\). Then
\(b_\ell\) is \(s\)-definable, \(X_n = \Uu_n/F_\Uu^d G_{n+1}\) is
\(cb_{<n}\)-definable and, if \(\ell > n\), the set \(X_\ell = b_{\ell}/F_\Uu^d
G_{n+1}\) is \(cb_{<n}b_{\ell-1}\)-definable. Also, \(X_\ell\) is
\(cb_{<\ell}\)-definably a (right) torsor for the group
\begin{align*}
F_\Uu^d G_\ell / F_\Uu^d G_{\ell+1}
&\isom G_\ell/(F_\Uu^d G_{\ell+1})\cap G_{\ell}\\
&\isom G_\ell/(F_\Uu^d\cap G_{\ell}) G_{\ell+1}\\
&\isom \K/p_{\ell}(F_\Uu^d\cap G_{\ell})\\
&\isom \K/d_id_j^{-1}(\I_{c_i}:\I_{c_j}),
\end{align*}
where \(\ell = (i,j)\) and the last isomorphism follows from \cref{descr
stabilizer} --- these isomorphisms are \(cb_{<n}\)-definable. Let \(\I_\ell =
d_id_j^{-1}(\I_{c_i}:\I_{c_j})\), then \(\I_\ell\) only depends on \(b_i =
v_{\Delta_i}(d_i)\) and \(b_j = \val_{\Delta_j}(d_j)\) and thus is
\(b_{<n}\)-definable. Since \(b_{n(n+1)/2 - 1} = uF^d_U\), we have \(s\in
\dcl(c,b_{<n},b_{n(n+1)/2-1})\). This concludes the first part of the
proposition.

Let us now fix \(a_{\ell} \in b_{\ell}\), for all \(\ell < n(n+1)/2\). Note that
if \(\ell < n\), \(v(a_\ell) = \val_{\Delta_\ell}(d_\ell)\) and we may assume that
\(d = a_{<n}\). If \(\ell = n\), as we saw above \(X_n = \Uu_n/F_\Uu^d G_{n+1}\)
is \(c a_{<n}\)-definably isomorphic to \(K/\I_{n}\). If \(\ell > n\), since
\(a_{\ell-1} \in b_{\ell-1} = u F_\Uu^d G_{\ell}\), we have \(a_{\ell-1} F_\Uu^d
G_{\ell+1} \in X_\ell\). This gives rise to a \(c a_{<n} a_{\ell-1}\)-definable
isomorphism \[f_\ell : X_\ell\isom F_\Uu^d G_\ell / F_\Uu^d G_{\ell+1} \isom
G_\ell/(F_\Uu^d G_{\ell+1})\cap G_{\ell} \isom \K/\I_\ell,\] where the first
isomorphism is induced by left multiplication by \(a_{\ell-1}^{-1}\) and the
third by the coordinate projection \(p_\ell\). Let \(h_\ell : \K \to H_\ell\) be
the section of \(p_\ell\). Then for every \(x\in f_\ell(b_\ell)\), since
\(p_\ell(h_\ell(x)) I_\ell = f_\ell(b_\ell)\), we have \(h_\ell(x) \in
a_{\ell-1}^{-1}u F_\Uu^d G_{\ell+1}\) and hence \(g(x) = a_{\ell-1} s(x) \in
b_\ell\).
\end{proof}

Recall that \(\mu_c : B_n \to \Mod_c\) is the canonical projection.

\begin{remark}
\label{decomp solv points}
Looking at the proof, all the operation applied to any upper triangular matrix
representation of \(s\) are actually field operations. It follows that
\cref{decomp solv} can be refined as follows. If \(A\leq\K(M)\) is a subfield, then:
\begin{itemize}
\item if \(s \in \mu_c(\B_n(A))\), then for every \(\ell\),
\(b_\ell(A)\neq\emptyset\);
\item if, for all \(k < \ell\), the set \(b_k(A)\) is nonempty, then \(I_\ell\)
is a translate of a \(c\)-definable cut by an element of \(\v(A)\), we have
\(f_\ell(b_\ell)(A) \neq \emptyset\) and \(g_\ell\) sends \(f_\ell(b_\ell)(A)\)
to \(b_\ell(A)\);
\item if, for all \(\ell\), the set \(f_\ell(b_\ell)(A)\) is nonempty, then, by
induction \(b_\ell(A)\neq \emptyset\) and we can choose all \(a_\ell \in
b_\ell(A)\) --- in particular, we can choose \(d = a_{<n} \in D_n(A)\) and \(u \in
b_{n(n+1)/2 -1}(A)\), so \(s\in \mu_c(\B_n(A))\).
\end{itemize}
\end{remark}

We conclude this section with one of our main uses for \cref{decomp solv}:
characterizing parameter sets over which every definable module has a
(triangular) basis. Recall that~\(\gBalls\) is the (ind-)interpretable set of
definable generalized balls.

\begin{corollary}
\label{lift G}
Let \(A\subseteq \eq{M}\) and \(C\subseteq\K(M)\) be a subfield. Assume that:
\begin{enumerate}
\item For every \(A\)-definable convex subgroup \(\Delta\leq\Vg\), we have
\((\Vg/\Delta)(\dcleq(A)) \subseteq \val_\Delta(C)\).
\item For every \(b\in \gBalls(\dcleq(AC))\) whose cut is a \(\v(C)\)-translate
of an \(A\)-definable cut, we have \(b(C) \neq\emptyset\).
\end{enumerate}
Then, for every tuple \(c\) in \(\Cut\), \(\Mod_c(\dcl(A)) \subseteq
\mu_c(\B_{|c|}(C))\).
\end{corollary}

\begin{proof}
Let \(s\in \Mod_c(\dcleq(A))\) --- in particular \(c\in \Cut(\dcleq(A))\) ---
and \(b = (b_\ell)_\ell\) be as in \cref{decomp solv}. By \cref{decomp solv
points}, it suffices to show that, \(f_\ell(b_\ell)(C) \neq \emptyset\). For
\(\ell < n\), since \(b_\ell \in \Vg/\Delta_\ell(\dcl(A))\), this follows from
the first assumption. If \(\ell\geq n\), by induction, \(f_\ell(b_\ell)\in
\K/\I_\ell\) is an \(AC\)-definable generalized ball whose cut is a
\(\v(C)\)-translate of an \(A\)-definable cut; and we conclude by the second
assumption.
\end{proof}

\subsection{The geometric sorts}\label{geosorts}

The goal of this section is to further simplify the codes of modules to
something more akin to the geometric sorts of \cite{HHM-EI}. This will be
crucial to classify \(\k\)-internal sets, when the value group is non discrete,
in \cref{kint}. 

Let \(M\) be an (expansion of a) valued field in some language \(\L\) with
stably embedded residue field.

\begin{definition}
\label{geom sorts}
\begin{enumerate}
\item Let \(\aCut\) denote \(\pCut\setminus\{\Vg_{>\gamma} : \gamma\in\Vg\}\)
--- unless \(\Vg\) is discrete, in which case, \(\aCut = \pCut\). Any module of
type a tuple \(c\) in \(\aCut\) is said to be \(\m\)-avoiding. Let \(\Lat\) be
the collection of codes for all \(\m\)-avoiding modules; that is \(\Lat =
\coprod_{c\in \aCut} \Mod_c\).
\item For every \(\O\)-module \(R\), let \(\red(R)\) denote the \(\k\)-vector
space \(R/\m R\) and let \(\red_R : R \to \red(R)\) denote the canonical
projection. We also define \(\Tor = \coprod_{R\in \Lat} \red(R)\).
\item Let \(\G = \K\cup\Lat\cup\Tor\) be the (generalized) geometric sorts.
\end{enumerate}
\end{definition}

\begin{remark}
\label{m idem}
For every \(c\in\pCut\), we have \(\m \I_c = \I_c\) if and only if \(\Cc_c \neq
\Vg_{\geq \gamma}\) for each \(\gamma\in\Vg\). Indeed, if \(\Cc_c\) does not
have a minimal element, then, for every \(x\in \I_c\), there exists \(a \in
\I_{c}\) such that \(\v(a) < \v(x)\). Then \(x a^{-1} \in\m\) and hence \(x = x
a^{-1} a \in \m \I_{c}\).

It follows that if \(c\) is some tuple in \(\pCut\) and \(R\) is an
\(\O\)-module of type \(c\). Then \(\red(R)\) has dimension \(|\{i \mid
\exists\gamma\in\Vg,\, \Cc_{c_i} = \Vg_{\geq \gamma}\}|\) over \(\k\).
\end{remark}

\begin{lemma}
\label{buildinglat}
Let \(R \subseteq \K^n\) be an \(\O\)-module of type \(c\) for some tuple \(c\)
in \(\pCut\). Then there exists an \(\code{R}\)-definable \(\m\)-avoiding module
\(\overline{R}\) containing \(R\) and such that \(\m R = \m \overline{R}\). In
particular, \(\red(R)\subseteq \red(\overline{R})\) is a subspace and
\(\dcleq(\code{R}) = \dcleq(\code{\overline{R}},\code{\red(R)})\).
\end{lemma}

\begin{proof}
Let \(v_R\) be the valuation defined in the proof of \cref{uppertrian}. Recall
that \(R = \{x\in\K^n \mid \Vg_{\geq 0} \subseteq v_R(x)\}\) and let
\(\overline{R} = \{x\in\K^n \mid \Vg_{> 0} \subseteq v_R(x)\}\). If \(e_i\) is a
(triangular) basis such that \(R = \sum_i \I_{c_i}e_i\), then \(\overline{R} =
\sum_i \overline{I}_{c_i} e_i\). 

For every \(\gamma\in\Vg\), let \(\gamma\O = \{x\in \K\mid \v(x)\geq\gamma\}\)
and \(\gamma\m = \{x\in \K\mid \v(x) > \gamma\}\). We have \(\overline{\gamma\O}
= \gamma\O\) and hence \(\m\overline{\gamma\O} = \gamma\m\). If \(I_c\) is
neither \(\gamma\O\) nor \(\gamma\m\) for any \(\gamma\in\Vg\), then
\(\m\overline{I}_{c_i} = \m\I_{c_i} = \I_{c_i}\). Finally, if \(\Vg\) is dense
then \(\overline{\gamma\m} = \gamma\O\) and \(\m\overline{\gamma\m} =
\gamma\m\). It follows that \(\m \overline{R} = \m R\).

The last assertion follows from the fact that
\(\red_{\overline{R}}^{-1}(\red(R)) = R\).
\end{proof}

Let us now recall, following \cite[Lemma~2.6.4]{HHM}, how to code definable
subspaces of \(\red(\overline{R})\). The following abstract conditions were
isolated in \cite{Hru-Groupoid}.

\begin{proposition}
\label{code subspace}
Let \(M\) be some \(\L\)-structure, \(\k\) be some stably embedded \(\L\)-definable
field and \(\bigcup_{s} V_s\) be a collection of finite dimensional
\(\L\)-definable \(\k\)-vector spaces which
\begin{itemize}
\item is closed under tensors: for every \(s,r\), there is an
\(\L\)-definable injection from the interpretable set \(V_s\tensor V_r\)
into some \(V_t\);
\item is closed under duals: for every \(s\), there is an \(\L\)-definable
injection from the interpretable set \(\dual{V_s}\) into some \(V_t\);
\item has flags: For every \(s\), there exists \(r,t\) and a
\(\L\)-definable exact sequence \(0 \to V_r \to V_s \to V_t\to 0\), with \(\dim(V_r) = 1\).
\end{itemize}
Then any definable subspace \(W\subseteq V_s\), is coded in \(\bigcup_{s} V_s\).
\end{proposition}

\begin{proof}
By (the proof of) \cite[Proposition~5.2]{Hru-Groupoid}, \(W\) is coded in some
projective space \(\PP(V_r)\). By \cite[Lemma~5.6]{Hru-Groupoid}, given that the
family has flags, \(\PP(V_r)\) is coded in \(\bigcup_{s} V_s\).
\end{proof}

We consider once again an (expansion of a) valued field \(M\) in some language \(\L\).

\begin{definition}
\label{LinA}
For every \(A \subseteq \eq{M}\), let \(\Lin_A = \coprod_{s\in\Lat(\acleq(A))}
\red(R_s)\) with its \(A\)-induced structure.
\end{definition}

Before we prove that \cref{code subspace} can be applied to \(\Lin_A\), let us
prove the following useful computation:

\begin{lemma}
\label{colon m}
Let \(I,J\leq\K\) be (\(M\)-definable) \(\O\)-submodules. If \((I:J) = \m\),
then \(I = a\m\) and \(J = a\O\), for some \(a \in \K\).
\end{lemma}

\begin{proof}
We first argue that $\v(I) \subsetneq \v(J)$. If $\v(J) \subseteq \v(I)$, then
$J \subseteq I$ so $\Stab(J) \subseteq (I:J)$, where $\Stab(J)=\{ x \in K \mid
xJ=J\}$. Since $\Stab(J)=\O_{\Delta}^{\times}$ for some definable convex
subgroup $\Delta$ of the value group, and $\m$ does not contain
$\O_{\Delta}^{\times}$ then $\v(I) \subsetneq \v(J)$.

We aim to show that $\v(J)$ has a minimal element. Otherwise, given  $\gamma
\in \v(J) \backslash \v(I)$ there is some $\beta \in \v(J)$ such that
$\beta<\gamma$, thus $\gamma=\beta+\delta$ where $\delta= \gamma-\beta>0$. Take
$x \in \m$ and $y \in J$ such that $v(x)=\delta$ and $v(y)=\beta$.
Consequently, $xy \notin I$ since $v(xy)=\gamma$, so $(I:J)\neq \m$.
Let $\gamma_{0}$ be the minimal element of $\v(J)$ and $a \in K$ such that
$v(a)=\gamma_{0}$ thus $J=a\O$. To show that $I=a\m$ it is
sufficient to argue that $\v(J) \backslash \v(I)=\{\gamma_{0}\}$. If there is
some $\beta \in \v(J) \backslash \v(I)$ such that $\beta \neq \gamma_{0}$, then
$\beta-\gamma_{0}>0$. Take $x \in \m$ such that
$v(x)=\beta-\gamma_{0}$, then $x \in \m$ but $ax \notin I$, hence
$(J:I) \neq \m$. Thus, $I=a\m$, as required.
\end{proof}

\begin{proposition}
\label{Lin closed}
For every \(A \subseteq \eq{M}\), \(\Lin_A\) is a collection of finite
dimensional \(\k\)-vector spaces which is closed under tensors, duals and has
flags.
\end{proposition}

\begin{proof}
Let \(R_1\subseteq\K^n\) and \(R_2\subseteq \K^m\) be two
\(\acleq(A)\)-definable \(\m\)-avoiding \(\O\)-modules. Let \(f :
\K^{n}\tensor\K^{n} \to \K^{nm}\) be the \(\L\)-definable isomorphism
induced by the canonical basis. By \cref{buildinglat}, we find an
\(A\)-definable \(\m\)-avoiding \(\O\)-module \(R = \overline{f(R_1\tensor
R_2)}\) inducing an inclusion \(\red(f(R_1\tensor R_2)) \subseteq \red(R)\).
Since \(\red(R_1)\tensor \red(R_2)\) is \(A\)-definably isomorphic to
\(\red(R_1\tensor R_2)\), we conclude that \(\Lin_A\) is closed under tensors.

As for duals, for every \(\O\)-submodule \(I\subseteq \K\), let \(\dual{I} =
(\O:I) = \{x\in\K\mid xI \subseteq \O\}\). By \cref{colon m},
\(\dual{I}\neq\m\). Let \(g : \dual{(\K^n)}\to\K^n\) be the
\(\L\)-definable isomorphism induced by the canonical basis composed with
the map \(\K^n\to\K^n\) reversing the order of the coordinates and let
\((\dual{a_i})_i\) denote the dual basis. Then, as we reversed the order,
\(g(\dual{a_i})\) is an upper triangular basis. If \(R = \sum_i \I_{c_i} a_i\)
for some triangular basis \((a_i)_i\), then \(g(\Hom_\O(R_1,\O)) = \sum_i
\dual{\I_{c_i}}g(\dual{a_i})\), which is \(\m\)-avoiding. This induces an
\(A\)-definable isomorphism \(\red(\Hom_\O(R_1,\O)) \isom
\red(g(\Hom_\O(R_1,\O)))\) showing that \(\Lin_A\) is closed under duals.

Finally, regarding flags, let \(R = \sum_i \I_{c_i} a_i\subset \K^n\) be an
\(\acleq(A)\)-definable \(\m\)-avoiding \(\O\)-module, with \(a\) triangular. We
find a flag for \(\red(R)\) by induction on \(n\). Let \(\pi\) be the projection
on the last \(n-1\) variables. Then, we have a split \(A\)-definable short exact
sequence
\[0 \to \I_{c_0}a_0 \to R \to \pi(R) \to 0\]
which induces the following \(A\)-definable short exact sequence
\[0 \to \I_{c_0}a_0/\m \I_{c_0}a_0 \to R/\m R \to \pi(R)/\m\pi(R) \to 0.\]
If \(\I_{c_0} = \m \I_{c_0}\), then \(\red(R) \isom \red(\pi(R))\) and we conclude by
induction on \(n\). If not, \(\I_{c_0}a_0/\m \I_{c_0}a_0\) is a dimension one
\(\k\)-vector space, and the above short exact sequence is a flag for
\(\red(R)\).
\end{proof}

\begin{remark}
\label{Linproperties}
Fix some \(A\subseteq \eq{M}\).
\begin{enumerate}
\item Recall that we assumed that \(\k\) is stably embedded in \(M\). It follows
that $\Lin_{A}$ is stably embedded and its \(A\)-induced structure is definable
with parameters in the structure with the \(\L\)-induced structure on
\(\k\) and the vector space structure on each sort \(R_s/\m R_s\). Indeed, once
we name a basis of every vector space in \(\Lin_A\), every definable subset in
\(\Lin_A\) can be identified with a definable subset in \(\k\). Similarly, if
\(\RV\) is stably embedded, so is \(\Lin_A\cup\RV\).
\item Whenever $\k$ is a pure algebraically closed field, combining \cite[Lemma
5.6]{Hru-Groupoid} with \cref{Lin closed}, \(\Lin_A\), with its \(A\)-induced
structure, eliminates imaginaries.
\end{enumerate}
\end{remark}

We can now improve \cref{code mod}:

\begin{corollary}
\label{code mod G}
Any \(M\)-definable \(\O\)-submodule \(R\) of type \(c\in\pCut\) is coded in
\(\G\).
\end{corollary}

Conversely, any element \(a + \m s \in \Tor\) is coded by the \(\O\)-submodule
generated by \((a + \m s)\times \{1\}\). So any element of \(\G\) is coded in
\(\K\cup\Mod\).

\begin{proof}
By \cref{buildinglat}, for some \(\m\)-avoiding module \(\overline{R}\)
containing \(R\) such that \(\m R = \m\overline{R}\), the lattice \(R\) is coded
by \(\code{\overline{R}} \in \Lat\) and \(\code{\red(R)}\) where \(\red(R)\leq
\red(\overline{R})\) is a subspace. By \cref{code subspace,Lin closed},
\(\red(R)\) is coded in \(\Lin_{\code{R}} \subseteq \Tor\).
\end{proof}

One of the main reason for isolating the \(\m\)-avoiding modules is the
following result.

\begin{proposition}
\label{orth Vg}
Let \(M\) be an \(\RV\)-expansion of a model of \(\Hen[0,0]\) with stably
embedded dense value group. If \(X\subseteq \Lat\) is \(M\)-definable and
orthogonal to \(\Vg\) --- meaning that any function \(f : X \to \eq{\Vg}\)
definable with parameters has finite image --- then \(X\) is finite. 
\end{proposition}

\begin{proof}
Let us first consider the case of some \(M\)-definable \(X\subseteq \K/\I_{c}\)
for some \(c\in\aCut\). Let \(Y\subseteq\K\) be the pre-image of \(X\). By
\cref{Hen prep}, there exists a (non-empty) finite set \(C\subseteq\K(M)\) such
that, any ball \(b\) disjoint from \(C\), is either contained in \(Y\) or is
disjoint from it. If \(X\) is infinite, then there exists some \(a\in Y\) such
that \(a+\I_c\) is disjoint from \(C\). Let \(b\) be the maximal ball around
\(a\) that is disjoint from \(C\) --- \emph{i.e.} the ball \(a + (a-c)\m\) where
\(v(a-c) = \min_{c\in C} \v(a-c)\). Then \(b\subseteq Y\) and \(b\) strictly
contains \(b\) as \(\I_{c}\neq d\m\), for any \(d\). It follows that the
function \(f : x\mapsto \v(x-a)\) induces a well-defined function on
\((b\setminus \{a+\I_c\})/\I_c \subseteq X\) with infinite image, contradicting
our hypothesis.

Let us now fix some \(e\in X \subseteq\Lat\), in some elementary extension of
\(M\). Let \(b_\ell\) be as in \cref{decomp solv} and let us prove, by induction
on \(\ell\), that \(b_\ell \in \eq{M}\). For all \(\ell < n\), we have
\(b_\ell\in \Vg/\Delta_\ell\), for some convex subgroup \(\Delta_\ell\). Since
\(e \in X\) and \(X\) is orthogonal to \(\Vg\) and \(b_\ell\in \dcleq(e)\), we
must have \(b_\ell\in \eq{M}\). Let now \(\ell \geq n\) and let us assume that
\(b_{<\ell} \in M\). We have \(b_\ell\in X_\ell\) which is an \(M\)-definable
torsor for some \(\K/\I_\ell\) where, by \cref{colon m}, \(\I_\ell\) is not a
multiple of \(\m\). Since \(b_i\in\dcleq(e)\), it is contained is an
\(M\)-definable subset of \(\K/\I_\ell\) which is orthogonal to \(\Vg\), and
hence, by the first paragraph, finite. So \(b_\ell\in \eq{M}\).

Since \(e\in\dcleq(M b)\), it follows that \(e\in \eq{M}\). As this holds for
any \(e\in X\) in some elementary extension, \(X\) is finite.
\end{proof}

Recall that a definable set \(X\) is (resp. almost) internal to another
definable set \(Y\) if, over a model, \(X\) admits a one-to-one (resp.
finite-to-one) map to \(\eq{Y}\).

\begin{corollary}
\label{almost k int}
Let \(M\) be an \(\RV\)-expansion of a model of \(\Hen[0,0]\) with dense value
group. Assume that \(\k\) and \(\Vg\) are stably embedded and orthogonal. Let
\(A=\acleq(A)\subseteq\eq{M}\) and let \(X\) be an almost \(\k\)-internal
\(A\)-definable subset of (some cartesian product of) \(\G\), then
\(X\subseteq \K(A)\cup\Lat(A)\cup\Lin_A\).
\end{corollary}

\begin{proof}
Any almost \(\k\)-internal set is orthogonal to \(\Vg\). By \cref{Hen prep}, any
infinite definable subset of \(\K\) contains a ball and hence is not orthogonal
to \(\Vg\), so if \(X\subseteq\K\), then \(X\subseteq\K(A)\). By \cref{orth Vg},
if \(X\subseteq\Lat\), then \(X\subseteq\Lat(A)\). Finally, if \(X\subseteq
\Tor\), then the projection of \(X\) to \(\Lat\) is finite and hence \(X\subset
\coprod_{s\in \Lat(A)}\red(R_s) = \Lin_A\).
\end{proof}

Note that it is necessary to assume that \(\Vg(M)\) is dense. The pure valued
field \(\mathbb{C}((t))\) eliminates imaginaries down to \(\G\). However, the
interpretable set \(\O/t^2\O\) is orthogonal to \(\Vg\) (in fact, it is
\(\k\)-analyzable), but it is not \(\k\)-internal; in particular, it cannot be a
subset of \(\K(A)\cup\Lat(A)\cup\Lin_A\).

\section{Density of quantifier free definable types}\label{densitysec}

Let \(M\) be a sufficiently saturated and homogeneous \(\RV\)-expansion of a
model of \(\Hen[0,0]\) with strongly stably embedded \(\Vg\) (\emph{cf}.
\cref{Vg k structure}). We realize types over \(\MM\) in some a sufficiently
saturated and homogeneous \(\MM\succ M\). Let \(M_1 = \ur{M}\) be its maximal
algebraic unramified extension (with the \(\L\)-induced structure on \(\Vg\))
and \(M_0 = \alg{M}\) be its algebraic closure (as a pure valued field). Note
that, by quantifier elimination, we have \(M_0 \prec \alg{\MM}\) and \(M_1\prec
\ur{\MM}\).

In what follows, whenever we want to refer to the structure in \(M_0 = \alg{M}\)
(resp. \(M_1 = \ur{M}\), resp. \(M\)), we will indicate this by a \(0\) (resp.
\(1\), resp. nothing): \emph{e.g.} \(\acl[0]\), \(\acl[1]\) or \(\acl\) for the
algebraic closure and \(\Tp^0(M)\) for the space of types in \(M_0\) over \(M\).

We also assume that the language \(\L_i\) of \(M_i\) is morleyized, and we
restrict ourselves to quantifier free \(\L_i\)-formulas when interpreting them
in a substructure. Note that it follows from quantifier elimination in \(M_1\)
(\cref{EQ Vg}) that the trace on \(M\) of any \(\L_1\)-definable set in \(M_1\)
is \(\L\)-definable. Similarly, for \(\L_0\)-definable sets. We also fix a
sufficiently saturated extension of the triple \((M_0,M_1,M)\) in which we
realize types.

The goal of this section is to prove the following density result:

\begin{theorem}
\label{density}
Assume that \(\Vg(M)\) satisfies Property \D. Let \(A = \acleq(A)\subseteq
\eq{M}\) and \(X\subseteq\K^n\) be non-empty and \(\L(A)\)-definable in \(M\).
Then there exists an \(\L(\G(A)\cup \eq{\Vg}(A))\)-definable type \(p \in
\Tp^1(M)\) consistent with \(X\).
\end{theorem}

This statement was proved in \cite[Theorem~5.9]{Vic-EIACk} when \(M = \ur{M}\)
and \(\Vg\) is an abelian ordered group of bounded regular rank. It also
strengthens \cite[Theorem~3.1.3]{HilRid-EIAKE} in two ways. The first is that it
provides a definable type in a stronger reduct (\emph{i.e.} a type in \(\ur{M}\)
and not just in \(\alg{M}\)). The second is that there is no hypothesis on the
residue field: it is not required that the residue field eliminates
\(\exists^\infty\).

\subsection{Codes of definable types}

We start by proving the following useful fact allowing to compare types in
\(M\), \(M_1 = \ur{M}\) and \(M_0 = \alg{M}\).

\begin{remark}\label{identification} Let $c=(c_{1},\dots, c_{n})$ be a finite
tuple in $\Cut^{**}$. Consider the natural map:
\begin{align*}
\B_{n}(M)/ \Stab(\Can_{c})(M) &\to \B_{n}(M_{1})/ \Stab(\Can_{c})(M_{1}).
\end{align*}

So the map above identifies the code of the $\mathcal{O}(M)$-module
$R=\sum_{i\leq n}a_{i}I_{c_{i}}(M)$ in $M$ with the code of the
$\mathcal{O}(M_{1})$-module \(R_1 = \sum_{i \leq n} a_{i} I_{c_{i}}(M_{1}) =
\sum_{i \leq n} a_{i} \O(M_1) \cdot I_{c_{i}}(M)\) generated by \(R\). Moreover,
we have \(R_1(M) = R\).
\end{remark}

Recall the \cref{def Delta type} of a definable \(\Delta\)-type.

\begin{proposition}
\label{code type}
Let \(A = \dcleq(A)\subseteq \eq{M}\) and \(p(x)\in\TP_{\K^{|x|}}^\epsilon(M)\)
be finitely satisfiable in \(M\) and \(\L(A)\)-definable, for \(\epsilon = 0\)
or \(1\). Then \(p\) has a unique extension
\(q_\epsilon\in\TP^\epsilon(M_\epsilon)\). Moreover, \(q_1\) is \(\L_1(\G(A)\cup
\eq{\Vg}(A))\)-definable and \(\restr{q_0}{M_1}\) is \(\L_1(\G(A))\)-definable.
\end{proposition}

We follow the proof of \cite[Theorem~5.9]{Vic-EIACk},
\emph{mutatis mutandis}.

\begin{proof} 
The uniqueness of \(q_\epsilon\) follows from \cite[Lemma~3.3.7]{HilRid-EIAKE};
note that \(q_\epsilon\) is finitely satisfiable in \(M\). 

For every integer \(d\geq 0\), let $V_{d} \simeq \K^\ell$ be the space of
polynomials in $\K(M)[x]_{\leq d}$ of degree less or equal than $d$ (in each
\(x_i\)). It comes with an \(\L(A)\)-definable valuation defined by \(v(P)\leq
v(Q)\) if \(p(x)\vdash \val(P(x))\leq\val(Q(x))\).

By \cref{uppertrian} there exists a separated basis \((P_i)_{i\leq \ell}\in
V_d\) such that, for every \(i\), \(\gamma_i  = v(P_i) \in \dcl(\code{v})
\subseteq A\). By \cite[Claim~3.3.5]{HilRid-EIAKE}, \((P_i)_{i}\) is also a
separated basis of \(V_d^1 = \K(M_1)[x]_{\leq d}\) with the valuation $v_{1}$
where \(v_1(P)\leq v_{1}(Q)\) if \(q_1(x)\vdash \val(P(x))\leq\val(Q(x))\). It
follows that \(v_1(V_d^1) = v(V_d) = \bigcup_{i} \gamma_i + \Vg(M)\) which we
identify in \(\eq{\Vg}\) over \(\G(A)\) with disjoint copies of \(\Vg\).

By \cref{code mod} the definable \(\O\)-modules \(R_i = \{P \in V_d \mid
v(P)\geq \gamma_i\}\) are coded by some tuple \(e_i\) in \(\Mod(A)\). We
identify \(e_i\) with the code \(e^1_i \in \Mod^1(M)\) of \(R^1_i = \{P \in
V^1_d \mid v_1(P)\geq \gamma_i\}\) via the map in \cref{identification}.
Furthermore, the \(R_i^1\) entirely determine \(v_1\) which is therefore coded
in \(\G(A)\) (\emph{cf.} \cref{code mod G}). Since \(\restr{q_0}{M_{1}}\) is
entirely determined by the valuations \(v_1\) on \(V_d^1\), the proposition is
proved in that case.

To conclude, let us prove the definability of \(q_1\). By \cref{EQ Vg}, any
\(\L_1\)-formula \(\phi(x,y)\) (with variables in \(\K\)) is equivalent to one
of the form \(\psi(\v(P(x,y)))\), where \(P\in \Zz[x,y]\) is a tuple. Let
\(X_\phi = \{v_1(P(x,a))\in v_1(V_d^1)\mid q_1(x)\vdash \psi(\v(P(x,a)))\}\).
Note that if \(v_1(P(x,b)) = v_1(P(x,a))\), then \(q_1(x)\vdash \v(P(x,b)) =
\v(P(x,a))\) and hence \(q_1(x)\vdash \psi(\v(P(x,b)))\) if and only if
\(v_1(P(x,b)) \in X_\phi\). So it suffices to show that the \(X_\phi\) are
\(\L_1(\G(A)\cup \eq{\Vg}(A))\)-definable, but this follows immediately from the
identification \(v_1(V_d^1) = \bigcup_{i} \gamma_i + \Vg(M)\).
\end{proof}

From now on, we will identify \(\L(A)\)-definable types in \(\TP^1(M)\) that are
finitely satisfiable in \(M\) with their unique extension to \(M_1 = \ur{M}\).

We now note that coding definable \(\L_0\)-types already allows us to code some
imaginaries, namely certain germs of functions into the space of balls ---
see \cref{imaginary} for the definition of germs.

\begin{definition}
Let \(B\) be a chain of balls in \(M_0\) (including points and \(\K\) itself).
We define the generic \(\eta_B\) of the generalized ball \(\bigcap_{b\in B} b\),
to be the \(\L_0(M_0)\)-type generated by:
\[\{x\in b\mid b\in B\}\cup \{x\nin b'\mid b'\text{ is a ball in }M_0\text{ and
}\forall b\in B,\ b'\text{ is a strict subball of }b\}.\]
\end{definition}

If \(b\) is a generalized ball, we write \(\eta_b\) for the type \(\eta_B\)
where \(B\) is the set of balls containing \(b\). If \(b\) is a definable
generalized ball in \(M\), \(\restr{\eta_b}{M}(x)\) is an
\(\L(\code{b})\)-definable type (finitely) consistent with \(x\in b\).

\begin{lemma}
\label{code germs balls}
Assume that \(\Vg(M)\) is dense. Let \(a\in\K(N)\), for some \(N\succ M\), be
such that \(p_0(x) = \tp_0(a/M)\) is \(\L(M)\)-definable. Let \(b(a)\) be an
\(\L_0(Ma)\)-definable open ball whose radius is in \(\Vg(N)\) --- where \(b\)
implicitly denotes some \(\L_0(M)\)-definable map. Then the germ
\(\germ{p_0}{b}\) is coded in \(\G(M)\) over \(\G(\code{p_0})\).
\end{lemma}

\begin{proof}
We may assume that \(\alg{N}\) is \(|M|^+\)-saturated. Let \(q_0(x,y)\) be the
\(\L(M)\)-definable \(\L_0(M)\)-type of \(ac\) where \(c\) is generic in
\(b(a)\) over \(Ma\) --- that is \(c\models \restr{\eta_b}{Ma}\). Note that,
since we are in equicharacteristic zero, \(b(a)\) has a point in \(N\) and, in
fact, since \(b(a)\) is open and \(\Vg(N)\) is dense, the generic of \(b(a)\) in
\(\alg{N}\) is satisfiable in \(N\). It follows that \(q_0\) is finitely
satisfiable in \(M\).

By \cref{code type}, \(p_0\) and \(q_0\) are coded in \(\G(M)\). Moreover, for
every \(\sigma\in\aut(M/\code{p})\) and \(ac \models q_0\), we have
\(\sigma(q_0) = q_0\) if and only if \(b(a) = b^\sigma(a)\). So \(\germ{p}{b}\)
is coded by \(\G(\code{q_0})\) over \(\G(\code{p_0})\).
\end{proof}

\subsection{Unary subsets of \texorpdfstring{\(\K\)}{K}}

We first consider the case of \cref{density} when \(X\subseteq\K\). The proof
proceeds as in \cite[Theorem~5.3]{Vic-EIACk} where, in the unary case, the
hypothesis that \(M = \ur{M}\) is not used.

\begin{lemma}
Let \(A = \acleq(A)\subseteq M\) and \(X\subseteq\K\) be \(\L(A)\)-definable in
\(M\). There exists an \(\L_1(\gBalls(A))\)-definable chain of balls \(E\) in
\(M_1\) such that the \(\L_1(\gBalls(A))\)-definable type \(\restr{\eta_E}{M_1}
\in \Tp^0(M_1)\) is consistent with \(X\).
\end{lemma}

This follows from \cite[Section 3]{HilRid-EIAKE} which states a relative version
of that statement. However, since the machinery set up for the relative version
of the statement is rather heavy, let us sketch a proof. A version of this proof
can also be found in \cite[Theorem 5.5]{Vic-EIACk}.

\begin{proof}
Let \(\Balls(M_1)\) be the set of all open and closed balls (including points
and \(\K\) itself) in \(M_1\). Given $b_{1},b_{2} \in \Balls$  we write \(b_1
\treeq b_2\) if \(b_1\cap X \subseteq b_2\cap X\). This is a pre-order with
associated equivalence \(\equiv\) and the associated order is a tree \(\cT\) if
we remove the class \(E_\emptyset\) of balls that don't intersect \(X\).

Note that any \(\equiv\)-class \(E\neq E_\emptyset\), the generalized ball \(b_E
= \bigcap_{b\in E} b\) is defined by knowing a point in \(b_E\) and the set
\(\{\val(x-y)\mid x,y\in b_E\}\) which is definable in \(\Vg(M) = \Vg(M_1)\). So
\(b_E\) and \(E\) are \(\L_1(M)\)-definable. It follows that \(E\) is coded in
\(\gBalls(M)\) and that the generic type \(\restr{\eta_{E}}{M_1}(x) \in
\TP^0(M_1)\) is \(\L_1(\code{E})\)-definable. If the type \(\eta_E\) is not
consistent with \(X\), by compactness, \(X\cap b_E\) is covered by finitely many
disjoint balls of \(M_0\). But then the smallest ball \(b_0\) in \(M_0\)
containing \(X\cap b_E\) is closed with radius in \(\Vg(M) = \Vg(M_1)\) and
\(X\cap b_E\) is covered by finitely many maximal open subballs of \(b_0\) which
are indeed balls of \(M_1\). It follows that \(E\) has finitely many direct
predecessors for $\treeq$, each of them in $\acleq(A, \code{E})$.

Starting with the class of \(\K\) and proceeding by induction, either the lemma
holds or the tree \(\cT\) has an initial infinite discrete finitely branching
tree. All the elements of this initial tree are \(\L_1(\gBalls(A))\)-definable
\(\equiv\)-classes.

By \cref{Hen prep}, there exists a finite set \(C\subseteq\K(M)\) preparing
\(X\) and can find an \(\equiv\)-class \(E\) in the infinite initial discrete
tree such that \(b_E\cap C = \emptyset\). Then \(b_E\subseteq X\) and hence
\(X\) is consistent with the \(\L_1(\gBalls(A))\)-definable type
\(\restr{\eta_E}{M_1}\).
\end{proof}

Let us now show that the type \(\restr{\eta_E}{M_1}\) can then be completed to
an \(\L_1(\gBalls(A)\cup\eq{\Vg}(A))\)-definable type \(p_1\in \TP^1(M_1)\)
consistent with \(X\):

\begin{lemma}
\label{compl L1}
Assume that \(\Vg(M)\) satisfies Property \D. Let \(A = \dcleq(A) \subseteq
\eq{M}\), let \(X\) be \(\L(A)\)-definable and let \(E\) be an
\(\L_1(\gBalls(A))\)-definable chain of balls in \(M_1\) such that
\(\restr{\eta_{E}}{M_1}\) is consistent with \(X\). Then, there exists an
\(\L_1(\gBalls(A)\cup\eq{\Vg}(\acl(A)))\)-definable type \(p \in \TP^1(M_1)\)
containing \(\restr{\eta_E}{M_1}\) and consistent with \(X\).
\end{lemma}

\begin{proof}
Let \(b_E = \bigcap_{b\in E} b\) and \(c\in X\) realize \(\restr{\eta_E}{M_1}\).
Then \(\gamma = \val(c-a)\) does not depend on \(a\in b_E(M_1)\). If \(\gamma
\in \Vg(M)\), then \(b_E\) is a closed ball and \(\restr{\eta_E}{M_1}\)
generates a complete type over \(M_1\).

If not, let \(C\) be the \(\L_1(\eq{\Vg}(A))\)-definable cut of \(\gamma\) over
\(\Vg(M_1)\). For every \(a\in b_E(M_1)\), let \(Y_a =\val(X-a) = \{v(x-a) \mid
x \in x\}\) --- it contains \(\gamma\).  Let \(\eta_C(x) = \{x < \gamma \mid
\gamma\in C\}\cup\{x > \gamma\mid \gamma\nin C\}\) be the
\(\L_1(\eq{\Vg}(A))\)-definable generic type of \(C\) over \(\Vg(M_1)\) and let
\(r(x) = \eta_C(x) \cup \{x\in Y_a \mid a\in b_E(M_1)\}\). Then \(r\) is
\(\L_1(\eq{\Vg}(A))\)-definable and \(\gamma\models r\). By  property \D, \(r\)
is contained an \(\L_1(\eq{\Vg}(A))\)-definable complete type \(q(x)\) over
\(\Vg(M_1)\).

Let \(\delta\models q\) and fix some \(a\in b_E(M_1)\). Since \(\delta\in Y_a\),
there exists \(c\in X\) such that \(\val(c-a) = \delta\), and since \(\delta
\models \eta_C\), we have \(c\models\restr{\eta_E}{M_1}\). So the type \(p(x) =
\restr{\eta_E}{M_1}(x) \cup q(\val(x-a))\) is consistent with \(X\). It is
complete by \cref{EQ Vg}, and it is
\(\L_1(\gBalls(A)\cup\eq{\Vg}(A))\)-definable.
\end{proof}

\begin{corollary}
\label{unary dens}
Assume that \(\Vg(M)\) satisfies Property \D. Let \(A = \acleq(A) \subseteq
\eq{M}\). Then any \(\L(A)\)-definable subset of \(\K\) is consistent with an
\(\L(\gBalls(A)\cup\eq{\Vg}(A))\)-definable type \(p \in \Tp^1(M)\). 
\end{corollary}

\subsection{Germs of functions}\label{germs dens}

To prove density of definable types in general, we now wish to proceed by
transitivity. However, since we are working with definable \(\L_1\)-types, we
first need to address the potential difference between \(\acl\) and \(\acl[1]\).
For every tuple \(a\) in \(\eq{M}\), let \(a_\K\) enumerate \(a\cap\K\).

\begin{proposition}
\label{ind:def acl}
Assume that \(\Vg\) has property \D. Let \(A = \acl(A)\subseteq \eq{M}\) and let
\(X\) be pro-\(\L(A)\)-definable in \(M\) in the sorts \(\G\cup\eq{\Vg}\). Let
\(p(x)\) be an infinitary \(\L(A)\)-definable \(\L_1(M)\)-type which is
consistent with \(X\). Assume that for any \(a\models p\), we have \(a \subseteq
\acleq[1](Ma_{\K})\). Then, for every \(\L(A)\)-definable one-to-finite
correspondence \(F\) into \(\G\cup\eq{\Vg}\) defined at all realizations of
\(p\) in \(X\), there exists an \(\L(A)\)-definable \(q(xy)\in\TP^1(M)\)
containing \(p(x)\) and there exists \(ac\models q\) with \(a\in X\) such that
\(c\in \acleq[1](M a_\K c_\K)\) and \(F(a) \cap \dcleq[1](ac) \neq \emptyset\).
\end{proposition}

Note that if \(a\in \G(M)\cup\eq{\Vg}(M)\), then as the induced
\(\L_1\)-structure on \(M\) is \(\L\)-definable, then \(\acl_1(a) \cap
(\G(M)\cup\eq{\Vg}(M))\subseteq \acl(a)\). Furthermore, if \(a\subseteq \K(M)\),
\(\K(\acl(a)) = \alg{\Qq(a)}\cap \K(M) \subseteq \K(\acl_1(a))\).

\begin{proof} 
If the proposition holds for \(F\), then for some \(\L_1\)-definable map, we
have \(f(ac) \in F(a)\). So replacing \(p\) by \(q\) and \(X\) by \(x\in X
\wedge f(xy)\in F(x)\), the hypothesis of the proposition stills holds and we
may assume that if \(a \in X\) realizes \(p\) then \(F(a)\cap\dcl_1(a)
\neq\emptyset\). The proof now proceeds by proving the proposition for various
specific \(F\) and changing \(p\) and \(X\) as above, until we exhaust all
possible \(F\). To do so, it suffices, by \cref{lift G}, to add a point to
\(a_K\) in every \(Aa\)-algebraic generalized ball and a lift of every element
in \(\Gamma/\Delta(\dcl(Aa))\), for any convex \(\Delta\leq\Gamma\).

We first consider the case where \(F\) is almost definable in
\(M_1\).
\begin{claim}\label{L1 germ}
If there exists an \(\L_1(M)\)-definable one-to-finite correspondence \(G\) such
that \(p(x) \wedge x\in X \vdash F(x)\subseteq G(x)\), then the proposition
holds.
\end{claim}

\begin{proof}
Since \(p\in \Tp^1(M)\), the cardinal of \(G(x)\) is constant when \(x\) varies
over realizations of \(p\); and we may assume that it is minimal. Then for every
other such \(G'\), we have \(p(x)\wedge x\in X\vdash F(x) \subseteq G(x)\cap
G'(x)\) and hence \(p(x) \wedge x\in X \vdash G(x) = G'(x)\). In other words,
for some (and hence for every) \(a \models p\), \(G(a) = G'(a)\). This holds in
particular of any \(G' = \sigma(G)\), where \(\sigma\in\aut(M/A)\) and thus
\(\germ{p}{G} \in \dcleq(A) = A\) --- recall that we assumed \(M\) to be
sufficiently saturated and homogeneous. Let \(G_0\subseteq G\) be minimal
\(\L_1(M)\)-definable such that \(p(x) \vdash \emptyset \neq G_0(x) \subseteq
G(x)\). Then \(\germ{p}{G_0} \in\acleq(A) = A\). The type \(q(x y) = p(x)\wedge
y\in G_0(x)\) is complete as cardinality of \(G_0\) is minimal and it is
\(\L(M)\)-definable and \(\aut(M/A)\)-invariant; so it is \(\L(A)\)-definable.
By construction, for any \(ac\models q\), we have \(c\in G_0(a) \subseteq
\acleq[1](Ma) \subseteq \acleq[1](M a_K)\). Moreover, if \(q(xy)\) is not
consistent with \(x\in X\wedge y\in F(x)\), then \(p(x) \wedge x\in X \vdash
F(x)\subseteq G(x)\setminus G_0(x)\), contradicting the minimality of \(G\). So
the type \(q\) is as required.
\end{proof}

Let us now assume that the co-domain of \(F\) is \(\eq{\Vg}\). By strong stable
embeddedness of \(\Vg\) and \cref{acl Vg}, for every \(a\models p\),
\[\eq{\Vg}(\acleq(Aa)) \subseteq \eq{\Vg}(\acleq(Ma)) =
\eq{\Vg}(\acleq(Ma_{\K})) \subseteq \acleq[1](\v(M(a_\K)))\] and hence, by
compactness, there exists an \(\L_1(M)\)-definable one-to-finite correspondence
\(G\) such that \(p(x) \vdash F(x) \subseteq G(x)\). We now conclude with
\cref{L1 germ}. As indicated at the start of the proof, we may therefore assume
that \(a\) contains all of \(\eq{\Vg}(\acleq(Aa))\).

\begin{claim}
\label{gen lift ball}
Let \(b: p \to \gBalls\) be \(\L_1\)-definable. Then there exists an
\(\L(A)\)-definable type \(q(xy)\in\TP^1(M)\) finitely satisfiable in \(M\)
containing \(p(x)\cup\{y\models\restr{\eta_{b(x)}}{Mx}\}\).
\end{claim}

\begin{proof}
Let \(a\models p\). Since \(\eq{\Vg}(\acleq_1(b(a))) \subseteq
\eq{\Vg}(\acl_1(a)) \subseteq \eq{\Vg}(\acl(a))\subseteq a\), by \cref{compl L1}
applied with \(M\) elementary extension of \(M_1\) containing \(a\), there exists an
\(\L_1(a)\)-definable type \(r_a(y)\) containing \(\eta_{b(a)}(y)\) and \(y\in
b(a)\). Let \(c\models \restr{r_a}{Ma}\). Then \(\tp_1(ac/M)\) is as required.
\end{proof}

For every \(\L(Aa)\)-definable convex subgroup \(\Delta\leq\Vg\) and \(\gamma
\in (\Vg/\Delta)(\acleq(Aa))\subseteq a\), the set
\(\val_\Delta^{-1}([\gamma,\infty])\) is an \(\L_1(a)\)-definable generalized
ball. Let \(q\) be as in \cref{gen lift ball} applied to any \(\L\)-definable
function \(b\) such that \(b(a) = \val_\Delta^{-1}([\gamma,\infty])\). If
\(ac\models q\), we have \(\val_\Delta(c) = \gamma\). Replacing \(p\) by \(q\),
we may assume that \(\val_\Delta(a_\K)\) contains all of
\((\Vg/\Delta)(\acleq(Aa))\).

If the co-domain of \(F\) is \(\k\), let \(\eta_k(y) \in \TP^1(M)\) be the
\(\L\)-definable generic of \(\k\) --- that is, the only non-algebraic type
concentrating on \(\k\). If \(p(x)\tensor \eta_\k(y)\) is not consistent with
\(x\in X \wedge y \in F(y)\), then there exists an \(\L_1(M)\)-definable
one-to-finite correspondence \(G\) such that \(p(x) \wedge x\in X \vdash F(x)
\subseteq G(x)\), by \cref{L1 germ} and \cref{gen lift ball}, we find \(q(xy)\)
consistent with \(x\in X \wedge \res(y) \in F(x)\). On the other hand, if
\(p(x)\tensor \eta_\k(y)\) is consistent with \(x\in X \wedge y \in F(x)\), let
\(q(x z) = p(x)\tensor \restr{\eta_\O(z)}{M}\), then, by hypothesis, \(q(xz)\)
is consistent with \(x\in X \wedge \res(z) \in F(x)\). Changing \(X\) and \(p\)
we may therefore assume that \(\k(\acleq(Aa))\subseteq \res(a_\K)\).

We may also assume that \(a_\K\) is a field. Now, if \(\xi\in \RV(\acleq(Aa))\),
then \(\val(\xi) \in \Vg(\acl(Aa)) \subseteq \val(a_\K)\). Let \(c\in a\) be
such that \(\val(c) =\val(\xi)\), then \(\xi\rv(c)^{-1} \in \k(\acl(Aa))
\subseteq \res(a_\K)\). It follows that \(\RV(\acleq(Aa))\subseteq \rv(a_\K)\).

Let us now assume that the co-domain of \(F\) is the set of generalized balls
that are not open balls. For any \(a\models p\) and \(b\in F(a) \subseteq
\acleq(Aa) \subseteq \acleq(Ma_\K)\), by \cref{points balls}, we have
\(b(\acl_1(Ma_\K)) \supseteq b(\acl(Ma_\K)) \neq \emptyset\). Moreover, the
radius of \(b\) is in \(\eq{\Vg}(\acleq(Aa))\subseteq \acl[1](Ma_\K)\), by
strong stable embeddedness. Hence, \(b \in \acleq[1](Ma)\) and we conclude with
\cref{L1 germ}. So we may assume that \(b\in a\). Applying \cref{gen lift ball},
we may further assume that \(b\) has a point in \(a\).

Now, if \(b\in \acleq(Aa)\) is an open ball, the smallest closed ball around
\(b\) has a point \(c \in a\) and \(b-c \in \RV(\acleq(Aa)) \subseteq
\rv(a_\K)\) also has a point in \(a\), hence so does \(b\). Recall that we
already assumed that, for every \(\acl(Aa)\)-definable convex subgroup
\(\Delta\leq\Vg\), \(\Vg/\Delta(\acleq(Aa)) \subseteq \val_\Delta(a_\K)\). By
\cref{lift G}, we have \(\G(\acleq(Aa)) \subseteq \dcleq[1](a)\).

Recall that to get to that point, we have replaced \(p(x)\) by a type \(q(xy)\)
consistent with \(X\) and such that if \(ac\models q\), then \(c\in \acl(Ma_\K
c_\K)\). So the proposition is proved.
\end{proof}

Applying \cref{ind:def acl} to all possible \(F\) --- as we have actually done
in the proof of the proposition --- we get:

\begin{corollary}
\label{def acl}
Let \(A = \acleq(A) \subseteq \eq{M}\), let \(X\) be \(\L(A)\)-definable and let
\(p \in \Tp^1(M)\) be \(\L(A)\)-definable and consistent with \(X\). Then there
exists \(a\models p\) in \(X\) such that
\[\tp_1(\G(\acleq(Aa))\cup\eq{\Vg}(\acleq(Aa))/M)\] is \(\L(A)\)-definable.
\end{corollary}

The proof of \cref{density} is now a standard induction.

\begin{proof}[Proof of \cref{density}]
Recall that, by \cref{code type}, we can identify \(\L(A)\)-definable types in
\(\TP^1(M)\) which are finitely satisfiable in \(M\) with their unique
\(\L_1(\G(A)\cup\eq{\Vg}(A))\)-definable extension to \(M_1\). Let \(N\prec M\)
be small and contain \(A\). Note that by \cref{EQ Vg}, \(N_1 = \ur{N} \prec
M_1\). So it suffices to prove the theorem over \(N\).

We proceed by induction on the integer \(n\) such that \(X\subseteq \K^n\). Let
\(\pi : \K^n \to \K\) be the projection on the first coordinate. By \cref{unary
dens}, there exists an \(\L(A)\)-definable type \(p\in\TP^1(M)\) consistent with
\(\pi(X)\).
Let \(a\in X\) realize \(\restr{p}{N}\) and let \(c\) enumerate
\(\G(\acleq(Aa))\cup\eq{\Vg}(\acleq(Aa))\). By \cref{def acl}, we may assume
that \(\tp_1(c/N)\) is \(\L(A)\)-definable.

By induction, there exists an \(\L_1(c)\)-definable \(r_c(z)\in\TP^1(M_1)\)
consistent with \(X_a\). In particular \(\restr{r_c}{M}\) is
\(\L(c)\)-definable. Let \(d\in X_a\) realize \(\restr{r_c}{M}\). Then \(q =
\tp_1(acd/N)\) is \(\L(A)\)-definable and \(ad \in X\). But \(q\) is
\(\L(\G(A)\cup\eq{\Vg}(A))\)-definable, by \cref{code type}, concluding the
proof.
\end{proof}

In the case \(M = \ur{M}\), we can deduce a slight generalization of
\cite[Theorem 5.12]{Vic-EIACk}:

\begin{corollary}
\label{wei res ACF}
Assume that:
\begin{itemize}
\item \(\Vg(M)\) satisfies Property \D;
\item \(\k(M)\) is a pure algebraically closed field.
\end{itemize}
Then \(M\) weakly eliminates imaginaries down to $\G \cup \eq{\Vg}$.
\end{corollary}

\begin{proof}
Fix some \(e\in\eq{M}\). Let \(A = \acl(e)\) and let \(f\) be an
\(\L\)-definable map with domain \(\K^n\) and such that \(X = f^{-1}(e) \neq
\emptyset\). By \cref{density}, there exists an
\(\L(\G(A)\cup\eq{\Vg}(A))\)-definable type \(p(x)\in\TP^1(M)\) consistent with
\(X\). But since \(M=M_1\), this type is complete and we have \(p(x)\vdash f(x)
= e\). It follows that \(e\in \dcl(\G(A)\cup\eq{\Vg}(A))\) proving weak
elimination.
\end{proof}

\section{Invariant Extensions}
\label{invext}
In this section, we will consider the invariance of types over stably embedded
definable sets, \emph{e.g.} $\RV$. This gives rise to several notions of
invariance isolated in \cite[Section~4.2]{HilRid-EIAKE}.

Whenever $D = \bigcup_{i} D_{i}$ is ind-\(\L\)-definable and stably embedded, we
denote by $\eq{D}$ the ind-\(\L\)-interpretable union of all \(\L\)-interpretable
sets $X$ that admit an \(\L\)-definable surjection $\prod_{j} D_{i_{j}}
\rightarrow X$.

Let \(M\) be sufficiently saturated and homogeneous. Let \(C\subseteq M\) be
potentially large and let \(D\) be ind-\(\L\)-definable and stably embedded.
Let \(p\) be a partial type over \(M\) (closed under implication).

\begin{definition} We say that the type \(p\):
\begin{enumerate}
\item is $\Aut(M/C)$-invariant if for every $\sigma \in \Aut(M/C)$, we have
$p=\sigma(p)$;
\item has $\Aut(M/C)$-invariant $D$-germs if it is $\Aut(M/C)$-invariant and so
is the $p$-germ of every (relatively) $M$-definable map $f:p\rightarrow \eq{D}$;
\item is $\Aut(M/D)$-invariant if it has $\Aut(M/D(M))$-invariant $D$-germs.
\end{enumerate}
\end{definition}

A nice property of the stronger notion is that it is transitive --- \emph{cf.}
\cite[Lemma~4.2.4]{HilRid-EIAKE}:

\begin{lemma}
\label{transitive}
Let \(N \succ M\) be saturated\footnote{We assume such models exist.} strictly
larger than \(M\). Let $p \in S(M)$ have $\Aut(M/C)$-invariant $D$-germs, let $a
\vDash p$ in $N$ and let $q \in S(N)$ be $\Aut(N/CD(N)a)$-invariant. Then
$\restr{q}{M}$ is $\Aut(M/C)$-invariant.

Moreover, if $q$ has $\Aut(N/CD(N)a)$-invariant $E$-germs, for some
(ind-)$\L$-definable set $E$, then $\restr{q}{M}$ has
$\Aut(M/C)$-invariant $E$-germs.
\end{lemma}

The main goal of this section is to prove the following statement. Recall
(\cref{LinA}) that \(\Lin_A = \coprod_{s\in\Lat(\acleq(A))} \red(R_s)\) and that
a bounded regular rank group is an ordered abelian group with countably many
definable convex subgroups.

\begin{theorem}
\label{invariantcompletions}
Let \(M\) be sufficiently saturated and homogeneous \(\RV\)-expansion of a model
of \(\Hen[0,0]\) such that the value group \(\Vg\) and residue field are stably
embedded and orthogonal and \(\Vg\) is \emph{either}:
\begin{itemize}
\item dense with property \D;
\item or, a pure discrete ordered abelian group of bounded regular rank --- in
that case we also add a constant for a uniformizer \(\pi\).
\end{itemize}
Let \(M_0 = \alg{M}\), let $A =\acl(A)\subseteq \eq{M}$ be small and let \(a\)
be a tuple in \(\K(N)\), for some \(N\succ M\). Assume that $\tp_{0}(a/M)$ is
$\aut(M/\G(A))$-invariant, then $\tp(a/M)$ is
$\Aut(M/\G(A),\RV(M),\Lin_{A}(M))$-invariant.
\end{theorem}

We follow the general strategy of \cite[Section~4]{HilRid-EIAKE}. The main new
challenge is to prove the equivalent (\cref{inv res}) of
\cite[Corollary~4.4.6]{HilRid-EIAKE} in the present setting since the geometric
sorts are now larger.

\subsection{Germs of functions into the linear sorts}\label{subsec: germs of functions in Lin}

One important ingredient of the proof of \cref{inv res} is a description of the
germ of certain functions into the linear sorts (\emph{cf}. \cref{code germs
dense,code germs discrete}). We proceed in three steps. First, we consider the
case of a valued field with algebraically closed residue field. Then we
consider valued fields with dense value groups (and arbitrary residue fields).
Finally, we consider valued fields with discrete value groups for which a
serious obstruction arises: the classification of \(\k\)-internal sets given in
\cref{almost k int} does not hold for discrete value groups. This can be
circumvented by considering a ramified extension with dense value group.

Let \(M\) be sufficiently saturated and homogeneous (\(\RV\)-expansion of a)
model of \(\Hen[0,0]\) in a language \(\L\),  whose value group \(\Vg\) is
stably embedded, has property \D and is orthogonal to \(\k\), which is itself
stably embedded.

We first prove that germs of functions into the linear sorts are internal to the
residue field (see \cref{germs intk}). We follow the general strategy of
\cite[Lemma~3.4.1]{HilRid-EIAKE} and first prove a result on the growth of
\(\dcl\) in \(\eq{\Lin_A}\). Recall the definition of an open generalized ball
(\cref{def gen ball}).

\begin{lemma}
\label{algebraic}
Let \(A\subset\eq{M}\) be small and let $U$ be an open $\infty$-\(A\)-definable
generalized ball. Let $a$ be a generic element of $U$ over $A$ in some \(N\succ
M\). Then $\eq{\Lin_{A}}(\dcl(Aa)) \subseteq \acl(A)$.
\end{lemma}

\begin{proof}
Let \(f : U \to \eq{\Lin_A}\) be some (relatively) $A$-definable function and
let \(\cut(U)\) be the cut of \(U\). Fix some \(e \in U(M)\). For now, we work
over \(M\), so we can assume that \(e = 0\) and we identify \(\eq{\Lin_A}\) with
\(\eq{\k}\). For every \(\gamma\in \cut(U)\) and \(d\in\eq{\k}\), let
\(X_{d,\gamma} = \{x\in U\mid \val(x) = \gamma\) and \(f(x) = d\}\). Then by
\cref{Hen prep}, there exists a finite set \(C\subseteq \K(M)\) which does not
depend on \(\gamma\) or \(d\), such that for every ball \(b\), if \(b\cap
C=\emptyset\), then \(b\cap X_{d,\gamma}\neq\emptyset\) implies that
\(b\subseteq X_{d,\gamma}\).

If \(a\in U\) is not in the smallest ball \(b\) containing \(C\cap U(M)\) and
\(0\), then the open ball of radius \(\val(a)\) around \(a\) --- that is
\(\rv(a)\) --- is entirely contained in \(X_{f(a),\val(a)}\). In other words,
\(f\) induces a well-defined function \(\overline{f} : \rv(U\setminus b) \to
\eq{\k}\).

\begin{claim}
Let \(f : \RV \to \eq{k}\) be \(M\)-definable. Then there are finitely many
\(\gamma_i\in\Vg(M)\) such that \(f(\{x\in \RV\mid \v(x)\neq \gamma_i\})\) is
finite.
\end{claim}

\begin{proof}
Let \(N\succ M\) be sufficiently saturated and homogeneous. For any choice of
\(\alpha \in\k\) and \(\gamma\in\Vg(N)\setminus\Vg(M)\), we find an automorphism
\(\sigma\in\aut(\RV(N)/\RV(M),\k(N))\) such that, if \(\v(x) = \gamma\), then
\(\sigma(x) = \alpha \cdot x\). First, as \(\k^\times\) is divisible and thus
injective, we find a group morphism \(h : \Vg\to\k^\times\) sending \(\gamma\)
to \(\alpha\) and \(\Vg(M)\) to \(0\). This \(h\) induces an automorphism
\(\sigma\in\aut(\RV(N)/\RV(M),\k(N))\) defined by \(\sigma(x) = h(\v(x)) \cdot x
\).

Let \(x,y\in \RV\) be such that \(\v(x) = \v(y)\nin\Vg(M)\). Then, by the above
paragraph, there is an automorphism \(\sigma\) fixing  \(\k\) and \(\RV(M)\),
and hence \(f\), and such that \(\sigma(x) = y\). It follows that \(f(x)
= \sigma(f(x)) = f(\sigma(x)) = f(y)\). By compactness, there are finitely many
\(\gamma_i\in\Vg(M)\) such that \(f\) induces a function \(\Vg\setminus\bigcup_i
\gamma_i \to \eq{\k}\). This function has finite image by orthogonality of
\(\k\) and \(\Vg\).
\end{proof}

Thus, we have found an \(M\)-definable closed ball \(b'\subset U\) such that
\(|f(U\setminus b')| = n <\infty\). By compactness, there exists an
\(A\)-definable \(Z\supseteq U\) such that \(|f(Z\setminus b')| = n\). If \(A =
M\), the proposition is proved. In general, we show that \(b'\) can be replaced
by a generalized \(A\)-definable ball. If there are two such \(M\)-definable
closed balls with empty intersection, then \(f(Z)\) is finite. If not, they form
a chain which is \(A\)-definable. Hence, their intersection is an
\(A\)-definable generalized sub-ball \(B\) of \(U\) such that \(f(U\setminus
B)\) is finite.

In both cases, if \(a\) is generic in \(U\) over \(A\), \(f(a)\) is in a finite
\(A\)-definable set. In other words, \(\eq{\Lin_A}(\dcl(Aa)) \subseteq \acl(A)\).
\end{proof}

\begin{proposition}
\label{growth acl Lin}
Let \(A\subseteq \eq{M}\) and let \(a\) be a tuple in \(\K(M)\). There is a
countable tuple $c \in \Lin_A(\dcl(Aa))$ such that
\[\eq{\Lin_{A}}(\dcl(Aa)) \subseteq \acl(Ac).\]
\end{proposition}

\begin{proof}
Let us first assume that \(|a| = 1\). Let $W=\{b \mid a \in b$ and $b$ is a
$A$-definable generalized ball$\}$. Then $a$ is generic over \(A\) in the
$\infty$-$A$-definable generalized ball
\[U=\bigcap_{b \in W}b.\]

If \(U\) is open, by \cref{algebraic}, we have \(\eq{\Lin_A}(\dcl(Aa)) \subseteq
\acl(A)\). If \(U\) is closed, let $c_0=\res_{U}(a) \in \eq{\Lin_A}(\dcl(Aa))$.
Let \(U_0\) be the intersection of all \(Ac_0\)-definable generalized balls
containing \(a\). Either \(U_0\) is open, or we set \(c_1 = \res_{U_0}(a)
\in \eq{\Lin_A}(\dcl(Aa))\). We continue this process unless \(U_i\) is open and we
set \(U_\omega = \bigcap_j U_j\). Then \(a\) is generic in the $Ac_{\geq
0}$-definable open generalized ball \(U_\omega\). By \cref{algebraic}, we have
$\eq{\Lin_A}(\dcl(Aa)) \subseteq \acl(A c_{\geq 0})$, concluding the proof.

Let us now assume that \(n > 1\) and proceed by induction. Let
\(d\in\eq{\Lin_A}(\dcl(Aa_{<n}))\) be such that \(\eq{\Lin_A}(\dcl(Aa_{<n}))
\subseteq \acl(Ad)\). By the case \(n=1\), let also \(c\in
\eq{\Lin_A}(\dcl(Aa))\) be such that \(\eq{\Lin_A}(\dcl(Aa)) \subseteq
\acl(Aa_{<n}c)\). For every \(e\in\eq{\Lin_A}(\dcl(Aa))\), there exists an
\(Aa_{<n}\)-definable one-to-finite correspondence \(f\) such that \(e \in
f(c)\). We have \(\code{f}\in\eq{\Lin_A}(\dcl(Aa_{<n})) \subseteq \acl(Ad)\) and
hence \(e \in \acl(Adc)\).
\end{proof}

\begin{lemma}
\label{germ int}
Let \(M\) be an \(\L\)-structure, let \(X\) and \(D\) be \(\L\)-definable sets
and let \(a\in X(N)\), for some \(N\succ M\). Assume there exists a countable
tuple \(c\) in \(M\) such that \(D(\dcl(Ma))\subseteq \acl(D(M)ac)\). Then, for
every \(\L\)-definable family \((f_\lambda)_{\lambda\in\Lambda} : X\to D\),
there exists an \(\L(M)\)-definable one-to-finite correspondence
\((g_\delta)_{\delta\in D^m} : X\to D\) such that, for every
\(\lambda\in\Lambda(M)\), there exists a \(\delta\in D^m(M)\) with
\(f_\lambda(a)\in g_\delta(a)\).

In particular, if \(p\in\TP(M)\) is definable, the interpretable set
\(\{\germ{p}{f_\lambda} \mid \lambda \in \Lambda\}\) is almost \(D\)-internal.
\end{lemma}

\begin{proof}
The existence of \(g\) follows by compactness, in a sufficiently saturated model
of the pair \((N,M)\). Now, if \(p\in\TP(M)\) is definable, for every
\(\lambda\in\Lambda(M)\), let \(Y_\lambda = \{\delta\in D^m\mid p(x) \vdash
f_\lambda(x) \in g_\delta(x)\}\). Then \(h(\germ{p}{f_\lambda}) =
\code{Y_\lambda}\) lies in an interpretable \(D\)-internal set. Moreover, since
\(p(x) \vdash f_\lambda(x) \in \bigcap_{\delta\in Y_\lambda} g_\delta(x)\) which
is finite, the map \(h\) is finite-to-one.
\end{proof}

\begin{proposition}
\label{germs intk}
Assume that \(\k\) is stable. Let $A \subseteq \eq{M}$ and let \(p(x)\in
\TP(M)\) be \(A\)-definable concentrating on \(\K^n\) for some \(n\). Let \(f\)
be an \(M\)-definable function. Assume that for every \(a\models p\), we have
\(f(a)\in\Lin_{Aa}\), then \(\germ{p}{f}\) lies in an almost \(\k\)-internal
\(A\)-definable set.
\end{proposition}

\begin{proof}
We may assume that \(A = \acl(A)\). As \(\k\) is stable, there exists a
\(\acl(Aa)\)-definable type \(q\) of bases for \(\Lin_{Aa}\). Let \(b\models
\restr{q}{Aa}\). Since \(\tp(\acl(Aa)/M)\) is \(A\)-definable, \(\tp(ab/M)\) is
also \(A\)-definable. Moreover, if \(g(ab)\in\k\) enumerates the coordinates of
\(f(a)\) in the basis \(b\), we have \(f(a) = h(b,g(ab))\) where \(h\) is
\(A\)-definable.

By \cref{growth acl Lin}, there exists a countable tuple \(c\) in \(M\) such
that \(\k(\dcl(Mab))\subseteq \acl(M \k(\dcl(abc)))\). As \(\k\) is stably
embedded, it follows that \(\k(\dcl(Mab))\subseteq \acl(\k(M)abc)\). Therefore, by
\cref{germ int}, the germ \(\germ{q}{g}\) lies in an \(A\)-interpretable almost
\(\k\)-internal set. Moreover, if \(\sigma\in\aut(M/A\germ{q}{g})\), then
\(\sigma(g)(ab) = g(ab)\) and hence
\[\sigma(f)(a) = h(b,\sigma(g(ab))) = h(c,g(ab)) = f(a).\] So \(\germ{p}{f}
\in\dcl(A\germ{q}{g})\) also lies in an \(A\)-interpretable almost
\(\k\)-internal set.
\end{proof}


\subsubsection{Dense value groups}\label{subsec: dense}

As previously, let \(M_0 = \alg{M}\) and \(M_1 = \ur{M}\).

\begin{lemma}
\label{dcl1 M}
We have \(\G(\dcl_1(M)) = \G(M)\).
\end{lemma}

\begin{proof}
Let \(A = \K(M)\). Note that \(\Vg(M_1) = \val(A)\). Also, any
\(\L_1(A)\)-definable ball \(b\) contains a point in \(a \in \K(M_1) \subseteq
\alg{A}\). Since \(A\) is henselian, the Galois-conjugates of \(a\) in \(M_1\)
over \(A\) are all in the generalized ball \(b\) and their mean \(d\) is fixed
by \(\Gal(M_1/A)\). Since the extension \(A \leq \K(M_1)\) is normal, \(d\in
A\). So we can apply \cref{lift G} (in \(M_1\)) to see that
\(\G(\dcl[1](M))\subseteq \bigcup_{c\in\Cut}\mu_c(\B_n(A)) = \G(M) \subseteq
\G(\dcl[1](M))\).
\end{proof}

If we further assume that the value group is dense, what we have done so far is
enough to show that germ of \(\L_0\)-definable open balls are coded in the
linear part.

\begin{lemma}
\label{code germs dense}
Assume \(\Vg(M)\) is dense. Let $A \subseteq \G(M)$ and let \(a\) be  a tuple of
\(\K\)-points in \(N\succ M\) be such that \(p = \tp_0(a/M)\) is
\(\L(A)\)-definable. Let \(b(a)\) be an open \(\L_0(Ma)\)-definable ball whose
radius is in \(\Vg(\dcl_1(Aa))\). Then, in the structure \(M_1\), the germ
\(\germ{p}{b}\) is coded in \(\G(\acl(A))\cup\Lin_A(M)\) over \(A\).
\end{lemma}

\begin{proof}
By \cref{code type} and property $\D$, we may assume that \(tp_1(a/M_1)\) is
\(\L_1(\acl_1(A))\)-definable. We have \(b(a) \in \Lin_{Aa}\), so, by
\cref{germs intk} applied in \(M_1\), the germ \(\germ{p}{b}\) lies in an
\(\L_1(A)\)-definable \(\k\)-internal set. On the other hand, by \cref{code
germs balls}, it is coded by some \(e\in\G(M_1)\) over \(A\). It now follows
from \cref{almost k int} that \(e\in\G(\acl_1(A))\cup \Lin_{\acl_1(A)}(M_1)\).
Since \(e \in\dcl_1(M)\), by \cref{dcl1 M}, we have \(e\in\G(M)\) and hence
\(e\in \G(\acl(A)) \cup \Lin_A(M)\).
\end{proof}

\subsubsection{Discrete value groups}\label{subsec: discrete}

We now assume that \(\Vg(M)\) is a pure, discrete ordered abelian group of
bounded regular rank. We add a constant \(\pi\) for a uniformizer in \(M\). We
introduce \(M_1' = M_1[\pi^{1/\infty}]\) the extension of \(M_1\) obtained by
adding \(n\)-th roots of \(\pi\) for all \(n > 0\). We assume the language
\(\L'_1\) of \(M_1'\) is morleyized and we restrict ourselves to quantifier free
\(\L'_1\)-formulas when interpreting them in a substructure. We write
\(\acl'_{1}\) and \(\dcl'_{1}\) to indicate the algebraic and definable closure
in $M_{1}'$.

\begin{lemma}
\label{lem: cuts are fine}
The definable convex subgroups of \(\Vg(M_1')\) are exactly the convex hulls of
definable convex subgroups of \(\Vg(M_1)\) and \(\Vg(M_1')\) has bounded regular
rank. Furthermore, the definable cuts in \(\Vg(M_1')\) are exactly the upward
closures of definable cuts in \(\Vg(M_1)\) and the cuts above or below a point
of \(\Vg(M_1')\).
\end{lemma}

\begin{proof}
Fix some \(n\in\omega\). Since $\Vg(M)$ has bounded regular rank, for each $n$
there is a finite sequence of convex subgroups $0< \Delta_{1} < \dots<
\Delta_{k}=\Vg(M)$ such that $\Delta_{i+1}/\Delta_{i}$ is $n$-regular. Note that
$\Vg(M_{1}')=\mathbb{Q} v(\pi)+ \Vg(M)$ and $\overline{\Delta}_{i}=\mathbb{Q} v(\pi)+
\Delta_{i}$ is also a convex subgroup of $\Vg(M_{1}')$. Then $\overline{\Delta}_{1}$
is $n$-divisible and $\overline{\Delta}_{i+1}/ \overline{\Delta}_{i}$ is $n$-regular as it
is isomorphic to $\Delta_{i+1}/ \Delta_{i}$. Consequently, for each $n< \omega$,
$\Vg(M_{1}')$ has the same $n$-regular rank than $\Vg(M)$, thus by
\cite[Proposition~2.3]{Farre} $\Vg(M_{1}')$ is of bounded regular rank and each
$\overline{\Delta}_{i}$ is definable in $\Vg(M_{1}')$. Furthermore, the map $\Delta$
to $\overline{\Delta}$ is a one to one correspondence between the convex subgroups of
$\Vg(M)$ and $\Vg(M_{1}')$.

Let $S \subseteq \Vg(M)$ be a definable cut and $\Delta_{S}= \{ \gamma \in
\Vg(M)\mid \gamma+S=S\}$. By \cite[Fact~3.2]{Vic-EIOAG} $\Delta_{S}$ is a convex
definable subgroup of $\Vg(M)$, and it is the maximal convex subgroup such that
$S$ is a union of $\Delta_{S}$-cosets. If \(\Delta_S = \{0\}\), then there
exists a \(\gamma\in S\) such that \(\gamma -\val(\pi) \nin S\). It follows that
\(S\) is the cut below \(\gamma\) and so is its upwards closure in
\(\Vg(M_1')\). If \(\Delta_S \neq\{1\}\), then \(S\) can be identified with a
subset of \(\Vg(M)/\Delta_S\) which is isomorphic to
\(\Vg(M_1')/\overline{\Delta}_S\) and hence the upwards closure of \(S\) in
\(\Vg(M_1')\) is definable.

Conversely, let $S' \subseteq \Vg(M'_1)$ be a definable cut and
$\overline{\Delta}_{S'}= \{ \gamma \in \Vg(M_1')\mid \gamma+S'=S'\}$. If
\(\overline{\Delta}_{S'} \neq\{0\}\), then, as above, \(S'\) is the upward
closure of \(S'\cap\Vg(M)\) which is definable. If \(\overline{\Delta}_{S'} =
\{0\}\), then, by \cite[Proposition~3.3]{Vic-EIOAG}, \(S'\) is of the form \(n x
\square \beta\) for some \(\beta \in \Vg(M_1')\) and \(\square \in \{>,\geq\}\).
Growing \(n\), we may assume that \(\beta\in\Vg(M)\). Moreover, since
\(\overline{\Delta}_{S'} = \{0\}\), for some \(\gamma\in S'\), \(\gamma -
\v(\pi) \nin S'\). As \((\gamma-\v(\pi),\gamma] \cap \Vg(M) \neq \emptyset\), we
may assume that \(\gamma\in\Vg(M)\). Then \(\beta = n\gamma - i\v(\pi)\), for
some \(i\), and \(S'\) is the cut above or below \(\gamma - n^{-1}i\v(\pi)\).
\end{proof}

It follows, by quantifier elimination in bounded regular rank group
(\cite[Theorem~2.17]{Vic-EIOAG}), that \(M_1'\prec \ur{\MM}[\pi^{1/\infty}]\).
Also, we can naturally identify the set \(\G(M_1)\) with a subset of
\(\G(M_{1}')\). We do, however, have to code the imaginaries of \(M\) that
\(M_1'\) believes to be geometric:

\begin{lemma}\label{dcl1' M}
Let \(R \in \Lat(\dcl[1]'(M_1))\).
\begin{enumerate}
\item There is a \(Q \in \Lat(M_1)\) such that $R \in \dcleq[1]'(Q)$ and $Q$ is
definable from \(R\) in the pair \((M_1',M_1)\).
\item For every \(e\in \red(R)(\dcl_1'(M_1))\), there exists \(\epsilon \in
\red(Q)(M_1)\) such that $e \in \dcleq[1]'(\epsilon)$ and $\epsilon$ is
definable from \(e\) in the pair \((M_1',M_1)\).
\end{enumerate}
\end{lemma}

\begin{proof}
For some \(n\geq 1\), we have \(R \in \mu_c(\B_m(\K(M_1)[\varpi]))\), where \(c\)
is a tuple in \(\aCut\) and \(\varpi^n = \pi\). Let \(f_\varpi : \K(M_1)^n \to
\K(M_1)[\varpi]\) send \(a\) to \(\sum_{i<n} a_i \varpi^i\). Then for every \(a\in
\K(M_1)^n\), \(\val(f_\varpi(a)) = \min_i \val(a_i) + n^{-1} i \val(\pi)\). It
follows that the pre-image of \(R(\K(M_1)[\varpi])\) by \(f_{\varpi}\) is an
\(\L(M_1)\)-definable \(\O\)-submodule \(Q(M_1)\) with \(Q\in \Lat(M_1)\). If
\(\varpi'\) is another \(n\)-th root of \(\pi\), then \(\varpi' =
\sigma(\varpi)\) for some \(\sigma\in\aut(M_1'/M_1)\). Then, since \(\sigma(R) =
R\), we have \[f_{\varpi'}^{-1}(R(\K(M_1)[\varpi'])) =
\sigma(f_{\varpi}^{-1}(R(\K(M_1)[\varpi]))) = \sigma(Q(M_1)) = Q(M_1).\] So \(Q\) does
not depend on the choice of \(\varpi'\) and it is definable from \(R\) in the
pair \((M_1',M_1)\). Also, since \(f_\varpi\) is linear, it induces a surjective
map \(Q(M_1') \to R(M_1')\), whose image does not depend on \(\varpi\). It
follows that \(R\in\dcl[1]'(Q)\).

Let us now consider some \(e\in \red(R)(\dcl_1'(M_1))\). Growing \(n\), we may
assume that \(e\in \red(R)(\K(M_1)[\varpi])\). Let \(\epsilon\) be the pre-image of
\(e\) under the bijection \(\red(Q)(M_1)\to\red(R)(\K(M_1)\varpi)\) induced by
\(f\). As above, \(\epsilon\) does not depend on the choice of \(\varpi\) and it
has the required properties.
\end{proof}

We can now prove a variant of \cref{code type} :

\begin{lemma}
\label{ext type M1}
Let \(A = \dcl(A)\subseteq \eq{M}\) and let \(a\in N\succ M\) be such that
\(\tp_0(a/M)\) is \(\L(A)\)-definable. Then \(\tp_0(a/M)\) has a unique
extension to \(\TP^0(M_1')\) and this extension is \(\L_1'(\G(A))\)-definable.
\end{lemma}

\begin{proof}
The uniqueness follows from \cref{code type} --- in fact, there is a unique
extension to \(M_0\). Let \(d \geq 0\), \(V_d = \K[x]_{\leq d}\) and \(v\) be
the valuation on \(V_d\) defined by \(v(P)\leq v(Q)\) if \(\v(P(a))\leq
\v(Q(a))\). By \cref{uppertrian}, the space \(V_d(M)\) admits a separated basis
\((P_i)_{i\leq \ell} \in V_d(M)\). By \cite[Claim~3.3.5]{HilRid-EIAKE}, it is
also a separated basis of \(V_d(M_1')\).

For every \(i,j\), let \(C_{i,j} = \{\gamma\in \Vg \mid v(P_i) + \gamma
v(P_j)\}\). If the stabilizer \(\Delta\) of \(C_{i,j}\) is not \(0\), then,
since \(\Vg/\Delta(M) = \Vg/\Delta(M_1')\) (by \cref{lem: cuts are fine}),
\(C_{i,j}(M)\) is co-initial in \(C_{i,j}(M_1')\) which is indeed definable. If
this stabilizer is \(0\), since \(\Vg(M_1)\) is discrete, \(C_{i,j}\) has a
minimal element \(\gamma_{i,j}\in\Vg(M)\) and \(v(P_j) = v(P_i)+\gamma_{i,j}\).
So \(v\) is indeed definable in \(M_1'\). Moreover, the \(\O_1'\)-module
\(R_i(M_1') = \{P\in V_d(M_1')\mid v(P) \geq v(P_i)\}\) is the \(\O_1'\)-module
generated by \(R_i(M) = \{P\in V_d(M)\mid v(P) \geq v(P_i)\}\) whose codes we
identify as in \cref{code type} via the natural inclusion map.
\end{proof}

We can now recover the equivalent of \cref{code germs dense} in the case of a
discrete value group:

\begin{lemma}
\label{code germs discrete}
Let $A \subseteq \G(M)$ and let \(a\) be a tuple of \(\K\)-points in \(N\succ
M\) such that \(p = \tp_0(a/M)\) is \(\L(A)\)-definable. Let \(b(a)\) be an open
\(\L_0(Ma)\)-definable ball whose radius is in \(\Vg(\dcl_{1}'(Aa))\). Then, in
the structure \(M_1\), the germ \(\germ{p}{b}\) is coded in
\(\G(\acl(A))\cup\Lin_A(M)\) over \(A\).
\end{lemma}

\begin{proof}
Given \cref{code type,dcl1 M}, we may assume that \(M=M_1\). Growing \(M_1\), we
may also assume that \(M_1\) is sufficiently saturated and homogeneous. By
\cref{ext type M1}, the type \(\tp_0(a/M_1')\) is \(\L_1'(A)\)-definable. Now,
applying \cref{code germs dense} in \(M_1'\), the germ \(\germ{p}{b}\) is coded
in \(\G(\acl_1'(A))\cup\Lin_{\acl_1'(A)}(M_1')\) over \(A\). In other words,
there are some tuple \(t\) in \(\K(\acl_1'(A))\cap \dcl_1'(M_1) = \K(\acl(A))\),
some \(R\in \Lat(\acl_1'(A))\cap \dcl_1'(M_1)\) and some
\(e\in\red(R)(\dcl_1'(M_1))\) which code \(\germ{p}{f}\) over \(A\). Let \(Q\) and
\(\epsilon\) be as in \cref{dcl1' M}. Now, any automorphism of
\(\sigma\in\aut(M_1/A)\) (extended in any way to \(M_1'\)) fixes \(Q\) if and only
if it fixes \(R\) --- so \(Q \in \Lat(\acl_1(A))\) --- and \(\sigma\) fixes
\(\germ{p}{f}\) if and only if it fixes \(t\), \(R\) and \(e\), if and only if
it fixes \(t\),\(Q\) and \(\epsilon\).
\end{proof}

\subsection{Invariant resolutions}

Let \(M\) be as in \cref{invariantcompletions}. As before, let \(M_0 = \alg{M}\)
and \(M_1 = \ur{M}\). Given a subset \(A\) of \(\G\), our goal is now to find a
subset \(C\) of \(\K\) whose type over \(M\) is invariant over \(A\) and some
stably embedded set. This will follows from the following lemma:

\begin{lemma}
\label{key}
Let \(N\succ M\), let $D \subseteq M$ be potentially large, let \(a\) be a tuple
in \(\K(N)\) and let $\rho$ be a pro-\(\L_1(M)\)-definable map. Assume that
$\rv(M(a)) \subseteq \dcl_1(D \rho(a))$ and that $p_1 = \tp_1(a/M)$ and
$\germ{p_1}{\rho}$ are $\Aut(M/D)$-invariant. Then $p = \tp(a/M)$ has
$\Aut(M/D)$-invariant \(\RV\)-germs. 
\end{lemma}

This is essentially \cite[Lemma~4.2.5]{HilRid-EIAKE} in a slightly different
context and the proof is identical. The main ingredient is elimination of
quantifier down to \(\RV\) --- see \cref{EQ RV}.

\begin{proof}
Let \(N_1\) be a large saturated elementary extension of \(M_1\) containing
\(N\). Fix $\sigma \in \Aut(M/D)$. Since $p_1=\tp_{1}(a/M)$ is $\Aut(M/D)$
invariant, there is an $\L_{1}$-elementary embedding $\tau: M(a) \rightarrow
M(a)$ extending $\sigma$. Because $\germ{p_1}{\rho}$ is $\Aut(M/D)$-invariant,
we have \(\rho(a)=\sigma(\rho)(a)\). Consequently, since $\rv(M(a)) \subseteq
\dcl[1](D\rho(a))$, \(\restr{\tau}{\rv(M(a))}\) is the identity map. By \cref{EQ
RV} in \(N_1\), extending \(\tau\) by the identity on \(\RV(N_1)\) yields an
$\L_{1}$-elementary embedding. Since \(\RV\) is stably embedded, this embedding
further extends to an element \(\tau\) of $\Aut(N_{1}/D,\RV(N_{1}),a)$ ---
\emph{cf.} \cite[Lemma~10.1.5]{TenZie}.

By \cref{EQ RV} (in \(M\) now), \(\restr{\tau}{M(a)\cup\RV(N)}\) is
\(\L\)-elementary. Consequently, $\tp(a,M)=\tp(a,\sigma(M))$ and we conclude that
$\sigma(p)=p$, as required. Lastly, we argue that $\tp(a/M)$ has
$\Aut(M/C)$-invariant \(\RV\)-germs. Let $X \subseteq \RV^{n}$ be
$\L(Ma)$-definable. Then, by \cref{EQ RV}, it is $\L(\rv(M(a)))$-definable and
hence $X(N)=\tau(X(N))=\tau(X)(N)$. Equivalently, $\sigma$ fixes the $p$-germ of
any $\L(M)$-definable function $f: p \rightarrow \eq{\RV}$.
\end{proof}

Let us now describe how \(\RV\) grows when adding one field element:

\begin{lemma}
\label{higherarity}
Let \(A\subseteq \G(M)\cup \eq{\Vg}(M)\) contain \(\G(\acl(A))\) and let \(a\)
be a tuple of \(\K\)-points in \(N\succ M\) such that \(p = \tp_1(a/M_1)\) is
\(\L(A)\)-definable. Let \(b(a)\) be an \(\L_1(Aa)\)-definable generalized ball.
If \(\Vg(M)\) is discrete, we assume that the cut of \(b(a)\) is
\(\L_{1}'(\G(A)a)\)-definable. Let \(c\in N\) realize the generic
\(\restr{\eta_{b(a)}}{M_1a}\) --- that is, \(c\) is in \(b(a)\) but not in any
proper generalized \(\acl_1(M_1a)\)-definable sub-ball. Let \(q =
\tp_1(ac/M_1)\). Then there is a pro-\(\L_1(M)\)-definable map \(\rho\) into
some power of \(\RV\) such that \(\germ{q}{\rho} \in \dcl(A,\Lin_{A}(M))\) and
\(\rv(M(ac))\subseteq \dcl[1](\rv(M(a)),\rho(ac))\).
\end{lemma}

\begin{proof}
We proceed by cases. 

\begin{claim}[{\cite[Lemma~4.3.10]{HilRid-EIAKE}}]
\label{4.3.10}
If \(b(a)\) is not closed, there is a pro-\(\L_0(M)\)-definable map \(\rho\)
into some power of \(\RV\) such that \(\germ{q}{\rho} \in \dcl(A)\) and
\(\rv(M(ac))\subseteq \dcl[0](\rv(M(a)),\rho(ac))\).
\end{claim}

Note that in equicharacteristic zero, condition (2) of \cite[Lemma
4.3.10]{HilRid-EIAKE} is verified as soon as \(b(a)\) is not closed. So we may
assume that \(b(a)\) is closed. By \cref{points balls}, there is an
\(\L_0(Ma)\)-definable finite set \(G(a) \subseteq \K\) such that \(G(a)\cap
b(a) = \{g\}\) is a singleton.

\begin{claim}[{\cite[Lemma~4.3.13]{HilRid-EIAKE}}]
\label{4.3.13}
If \(b(a)\) is closed and \(G(a) \subseteq \K\) is an \(\L_0(Ma)\)-definable
finite set such that \(G(a)\cap b(a) = \{g(a)\}\) is a singleton, then
\[\rv(M(ac)) \subseteq\dcl[1](\rv(M(a)),\rv(c- g(a))).\]
\end{claim}

Note that such a \(G(a)\) always exists by \cref{points balls}. Let \(\rho(ac) =
\rv(c- g)\in\dcl_0(Mac)\). We have \(\gamma = \val(c-g) \in \Vg(\dcl_1(Aa))\) as
it is the radius of \(b\). If \(\Vg(M)\) is discrete, it is in
\(\dcl[1]'(\G(A)a)\) by assumption. So, by \cref{code germs dense,code germs
discrete}, we have \(\germ{q}{h}\in \dcl[1](A\cup\Lin_A(M))\). Then, as
required, we have \(\germ{q}{\rho}\in\dcl[1](A,\Lin_{A}(M))\cap \G(M)\) and
\(\rv(M(ac)) \subseteq \dcl[1](\rv(M(a)),\rho(ac))\).
\end{proof}

We can now prove the existence of sufficiently invariant resolutions of
geometric points:

\begin{proposition}
\label{inv res}
Let \(A  = \acl(A)\subseteq \eq{M}\). There exists \(C\subseteq \K(N)\), for
some \(N\succ M\), with:
\begin{enumerate}
\item $\G(A) \subseteq \dcl[1](C,\Vg(M))$;
\item \(\tp_1(C/M)\) is \(\L(\G(A)\cup\eq{\Vg}(A))\)-definable;
\item \(\tp(C/M)\) has $\Aut(M/\G(A),\RV(M),\Lin_{A}(M))$-invariant $\RV$-germs.
\end{enumerate}
\end{proposition}

\begin{proof}
By transfinite induction, we construct a tuple \(c\) in \(N\succ M\) and a
(pro-)\(\L_1(M)\)-definable function $\rho$ such that:
\begin{itemize}
\item \(p_1=\tp_1(c/M)\) is finitely satisfiable in \(M\) and $\L(A)$-definable;
\item \(\rv(M(c))\subseteq \dcl[1](\RV(M),\rho(c))\);
\item \(\germ{p_1}{\rho} \in \dcl[1](\G(A),\RV(M),\Lin_A(M))\);
\item any $\L_{1}(\G(A)c)$-definable generalized ball \(b\) has a point in \(C\);
\item for all $\L_{1}(\G(A)c)$-definable convex subgroup $\Delta \leq \Vg$,
$\Vg/\Delta(\dcl[1](Ac))\subseteq \val_{\Delta}(\K(C))$.
\end{itemize}
Note that \(\tp_1(\acl_1(Ac)\cap N/M)\) is definable over
\(\G(\acl(A))\cup\eq{\Vg(\acl(A))}\subseteq A\) (\emph{cf.} \cref{code type}).
Let \(b(c)\) be an \(\L_1(\G(A)c)\)-definable generalized ball whose cut is a
\(\v(c)\)-translate of an \(\L(A)\)-definable cut. By property \D, the generic
\(\eta_{b(c)}\) can be extended to a complete \(\L(\acl_1(Ac)\cap N)\)-definable
\(\L_1\)-type --- and this type is finitely satisfiable in \(N\). Using
\cref{higherarity}, we can thus add a generic of \(b(c)\) to \(c\). We then
iterate this construction.

Given such a tuple \(c\), by \cref{lift G} applied in \(M_1\), we have
\(\G(A)\subseteq \dcl_1(c,\eq{\Vg(A)}) \subseteq \dcl_1(c,\Vg(M))\). Moreover,
the type \(\tp(c/M)\) has \(\aut(M/A,\RV(M),\Lin_A(M))\)-invariant \(\RV\)-germs
by \cref{key}.
\end{proof}

We deduce \cref{invariantcompletions} from \cref{inv res} and the machinery of
\cite[Section~4]{HilRid-EIAKE}.

\begin{proposition}[{\cite[Corollary~4.4.1]{HilRid-EIAKE}} and \cref{key}]
\label{4.4.1}
Let \(A\subseteq \K(M)\) and \(R\subseteq\RV(M)\). There exists \(C\subseteq
\K(N)\), for some \(N\succ M\), such that and \(\tp(C/M)\) is
$\Aut(M/A,\RV)$-invariant and \(R\subseteq \dcl_0(AC)\).
\end{proposition}

\begin{proposition}[{\cite[Corollary~4.4.3]{HilRid-EIAKE}}]
\label{4.4.3}
Let \(A\subseteq \K(M)\). There exists \(C\prec N\), for some \(N\succ M\), such
that and \(\tp(C/M)\) is $\Aut(M/A,\RV)$-invariant and \(C\) contains a
realization of every \(\L(A)\)-type.
\end{proposition}

\begin{proposition}[{\cite[Corollary~4.3.17]{HilRid-EIAKE}}]
\label{4.3.17}
Let \(A\subseteq\eq{M}\) be small, let \(C\prec M\) contain a realization of
every \(\L(A)\)-type and let \(a\in \K(N)\), for some \(N\succ M\) be such that
\(\tp_0(a/M)\) is \(\Aut(M/A)\)-invariant. Then \(\tp(a/M)\) is
\(\Aut(M/C,\RV)\)-invariant.
\end{proposition}

\begin{proof}[Proof of \cref{invariantcompletions}]
Fix \(A = \acl(A)\subseteq \eq{M}\) and \(a\) in some elementary extension of
\(M\) such that \(p_0 = \tp_0(a/M)\) is \(\aut(M/\G(A))\)-invariant.
\begin{itemize}
\item By \cref{inv res}, we find \(C\subseteq \K(N)\), for some \(N \succ M\)
(sufficiently saturated and homogeneous), such that \(\tp(C/M)\) has
$\Aut(M/\G(A),\RV(M),\Lin_{A}(M))$-invariant $\RV$-germs and $\G(A) \subseteq
\dcl[1](C\gamma)$ for some (infinite) tuple \(\gamma\) in \(\Vg(M)\).

\item By \cref{4.4.1} and transitivity (\cref{transitive}), growing \(C\), we
may assume that \(\gamma \in\val(C)\). By \cref{4.4.3} and transitivity, we can
further assume that \(C\prec N\) contains a realization of every type over
\(\G(A)\).

\item We may assume that \(a\models \restr{p_0}{N}\) --- \emph{cf.}
\cite[Claim~4.4.7]{HilRid-EIAKE}. Then \(\tp_0(a/N)\) is
\(\aut(N/C)\)-invariant and \(\tp(a/N)\) is
\(\aut(N/C,\RV)\)-invariant by \cref{4.3.17}.
\end{itemize}

By transitivity \(\tp(a/M)\) is \(\aut(M/\G(A),\RV(M),\Lin_A(M))\)-invariant.
\end{proof}

\section{Eliminating imaginaries}\label{conclusion}

Following the general strategy of \cite[Theorem~6.1.1]{HilRid-EIAKE}, we can now
deduce elimination of imaginaries. Let \(M\) be a sufficiently saturated and
homogeneous, and as in \cref{invariantcompletions}. Let \(M_0 = \alg{M}\) and
\(M_1 = \ur{M}\).

\begin{proposition}
\label{weakcode}
Let $e \in \eq{M}$ and $A=\acleq(e)$. Then 
\begin{equation*}
e \in \dcleq(\G(A), \eq{(\RV \cup \Lin_{A})}(A)). 
\end{equation*}
\end{proposition}

\begin{proof}
We may assume \(M\) is sufficiently saturated and homogeneous. There is an
$\L$-definable map $f$ and a tuple $m$ in $K(M)$ such that $f(m)=e$. Let
$X=f^{-1}(e)$. By \cref{density} we can find a type $p \in \TP^{1}(M)$ such
that:
\begin{itemize}
\item $p \cup X$ is consistent;
\item $p$ is $\L_{1}(\G(A) \cup \eq{\Vg}(A))$-definable.
\end{itemize}
Take $a\in X$ satisfying \(p\). Then $\tp_{0}(a/M)$ is \(\L(\G(A))\)-definable.
By \cref{invariantcompletions}, the type \(q = \tp(a/M)\) is
$\Aut(M/\G(A),\RV(M),\Lin_{A}(M))$-invariant. So, for every automorphism $\sigma
\in \Aut(M/\G(A),\RV(M),\Lin_{A}(M))$, we have \(e = \sigma(e)\) since \(q =
\sigma(q) \vdash \sigma(e) = f(x) = e\).

As \(\RV\cup\Lin_{A}\) is stably embedded (\emph{cf.} \cref{Linproperties}), it
follows (\emph{e.g.} \cite[Lemma~4.2.3]{HilRid-EIAKE}) that
\[e\in\dcl(\G(A),\RV(M),\Lin_{A}(M)).\]

So there is a $\L(\G(A))$-definable function $g$ and a tuple $c$ in $\RV^{m}(M)
\times \Lin_{A}^{n}(M)$ such that $g(c)=e$. Let $Y=g^{-1}(e)$. This is an
$A$-definable subset of $\RV^{m} \times \Lin_{A}^{n}$. Consequently,
$\code{Y} \in \eq{(\RV \cup \Lin_{A})}(A)$ and
\[e \in \dcleq(\G(A), \code{Y}) \subseteq \dcleq(\G(A) \cup \eq{(RV \cup
\Lin_{A})}(A)),\] as required.
\end{proof}

We now want to describe the imaginaries in \(\RV\cup\Lin_A\).

\begin{proposition}[{\cite[Proposition 5.3.1]{HilRid-EIAKE}}]
\label{imaRV}
Further assume that \(M\) is a \(\k\)-\(\Vg\)-expansion of \(\Hen[0,0]\) and
that for every $n \in \Zz_{\geq 2}$, one has $[\Vg: n\Vg]< \infty$ and the
pre-image in $\RV$ of any coset of $n\Vg$ contains a point which is algebraic
over $\emptyset$. Then, for $A=\acleq(A) \subseteq \eq{M}$, we have
\[\eq{(\RV \cup \Lin_{A})}(A) \subseteq \dcleq(\eq{\Vg}(A) \cup
\eq{\Lin_{A}}(A)).\]
\end{proposition}

Finally, let us relate \(\eq{\Lin_A}\) to the linear imaginaries \(\lineq{\k}\):

\begin{lemma}
\label{factLin}
Let $A=\dcleq(A) \subseteq M$. Then $\eq{\Lin_{A}}(A) \subseteq
\dcleq(\lineq{\k}(A))$.
\end{lemma}

\begin{proof}
Recall that \(\Lin_A\) is a stably embedded collection of \(\k\)-vector spaces
--- see \cref{Linproperties}. Take $e \in \eq{\Lin_{A}}(A)$. Then $e$ is the
code of a definable set \(X \subseteq \prod_i\red(R_i)\) where \(a\in
\prod_j\red(R'_j)\), the family \((X_a)_a\) is \(\L(A)\)-definable and \(R_i\)
and \(R'_j\) are $A$-definable $\m$-avoiding module. Then \(R = \prod
R_i\times\prod_j R'_j\) is an $A$-definable $\m$-avoiding module. Adding zero
coordinates, we may assume that we have \(X_a\subseteq \red(R)\) and \(a\in
\red(R)\).

For every basis \(b\) of \(\red(R)\), the set \(X_a\) is
\(\L_{\vect}(\eq{\k}(A)ab)\)-definable in \((\k,\red(R))\). Replacing \(a\) with
\(ab \in \red(R)^{d+1}\), we may assume that \(X_a\) is
\(\L_{\vect}(a)\)-definable. Let \(a E a'\) be the equivalence relation defined
by \(X_a = X_a'\) and let \(c\) be the type of \(R\). Then
\(e\in\dcl(\Tor_{c,V^{d+1}}(A))\).
\end{proof}

We can now prove our main results. Let \(M\) be as in
\cref{invariantcompletions}.

\begin{theorem}
\label{AKE}
Assume that \(M\) is a \(\k\)-\(\Vg\)-expansion of \(\Hen[0,0]\) and that, for
every $n \in \Zz_{\geq 2}$, one has $[\Vg: n\Vg]< \infty$ and the pre-image in
$\RV$ of any coset of $n\Vg$ contains a point which is algebraic over
$\emptyset$. Then $M$ weakly eliminates imaginaries down to $K \cup
\lineq{\k}\cup \eq{\Vg}$. 
\end{theorem}

\begin{proof}
Let $e \in \eq{M}$ and $A=\acleq(A)$. By \cref{weakcode}, we have \[e \in
\dcleq(\G(A) \cup \eq{(\RV \cup \Lin_{A})}(A)).\] By \cref{imaRV}, we have
\[\eq{(\RV \cup \Lin_{A})}(A) \subseteq \dcleq(\eq{\Vg}(A),
\eq{\Lin_{A}}(A))\subseteq \dcleq(\eq{\Vg}(A), \lineq{\k}(A)),\]
where the last inclusion follows from \cref{factLin}.
\end{proof}

\begin{remark}
\label{general RV im}
In general, imaginaries in the short exact sequence \(\k^\times\to
\RV^\times\to\Vg\) might be more complicated. We can embed the short exact
sequence \(\k^\times(M)\to \RV^\times(M)\to\Vg(M)\) into a sequence
\(\k^\times(M)\to H \to\Qq\tensor \Vg(M)\) --- this is clear after adding a
section in some \(\aleph_1\)-saturated elementary extension. A refinement of
\cite[Theorem~5.1.4]{HilRid-EIAKE} shows that any set \(X\) definable in
\(\RV(M)\) is weakly coded in \(\eq{\Vg}(M)\cup (\eq{H_{\Delta}} \cap
\dcl(M))\), where \(\Delta\) is the divisible hull of \(\Vg(\code{X})\).

The elements of \(\eq{H_{\Delta}} \cap \dcl(\RV)\) correspond to \(M\)-definable
subsets of \(H_\delta\) for some tuple \(\delta\) in \(\Delta\subseteq
\Qq\tensor\Vg(M)\). These are the imaginaries that we need to add, or classify,
to obtain a version of \cref{AKE} without any finiteness assumption on
\(\Vg/n\Vg\).

If $[\Vg: n\Vg]<\infty$ for every \(n \geq 2\) and we add constants as in
\cref{AKE}, then any \(H_\delta\), for \(\delta\in\Delta\), is canonically
isomorphic to some \(H_\gamma\) with \(\gamma\in\Vg(\code{X})\); and we recover
\cref{AKE}. In contrast, if \(\k\) is a pure algebraically closed field, the
\(\k\)-linear structure \(H_\Delta\) eliminates imaginaries and
\(H_\Delta\cap\dcl(M) = H_{\Vg(\code{X})}\), yielding back \cref{wei res ACF}.
\end{remark}

\begin{theorem}
\label{ac}
Assume that \(M\) admits \(\L\)-definable angular components. Then $M$ weakly
eliminates imaginaries down to $K \cup \lineq{\k} \cup \eq{\Vg}$.
\end{theorem}

\begin{proof}
Let $e \in \eq{M}$ and $A=\acleq(A)$. By \cref{weakcode}, $e \in
\dcleq(\G(A), (\RV \cup \eq{\Lin_{A}})(A))$. Since $\RV$ is
\(\L\)-definably isomorphic to $\k^{\times} \times \Vg$, then $\eq{(\RV
\cup \Lin_{A})} \subseteq \eq{(\Vg \cup \Lin_{A})}$. The statement now follows
from orthogonality of \(\Vg\) and \(\Lin_A\) and \cref{factLin}.
\end{proof}

As an illustration, we conclude this
paper with the complete classification of (almost) \(\k\)-internal sets, when the value group is dense.

\begin{corollary}
\label{kint}
Let $M$ be as in \cref{AKE} or \cref{ac} and assume that \(\Vg(M)\) is dense.
Let \(A\subseteq\eq{M}\) and \(X\) be \(A\)-definable. The following statements
are equivalent:
\begin{enumerate}
\item \(X\) is \(\k\)-internal;
\item \(X\) is almost \(\k\)-internal;
\item \(X\) is orthogonal to \(\Vg\);
\item \(X \subseteq \dcl(\acl(A),\Lin_A)\).
\end{enumerate}
\end{corollary}

\begin{proof}
The fourth statement is a particular case of the first statement. The second
statement is a particular case of the first, and it implies the third since
\(\k\) and \(\Vg\) are orthogonal. There remains to prove that if \(X\) is
orthogonal to \(\Vg\) then it is a subset of \(\acl(A)\cup\eq{\Lin_A}\). By
\cref{AKE,ac}, any element \(a\in X\) is weakly coded by some tuple \(\eta\) in
\(\K\cup\lineq{\k}\cup\eq{\Vg}\). Then \(\eta\) also lies on a \(A\)-definable
set orthogonal to \(\Vg\).

There remains to show that \(\eta \in \dcl(\acl(A)\cup\eq{\Lin_A})\). We may
assume that \(\eta\) is a single point. Since \(X\) is orthogonal to \(\Vg\), if
\(\eta \in \K\cup\eq{\Vg}\), then \(\eta\in\acl(A)\). If \(\eta\in\Lat\), then,
by \cref{orth Vg}, \(\eta\in\acl(A)\). Finally, if \(\eta\in \eq{\red(R_s)}\),
for some \(s\in\Lat\), then \(s\in\acl(A)\) and hence \(\eta\in\eq{\Lin_A}\).
\end{proof}

\sloppy
\printbibliography

\end{document}